\documentclass[twoside,leqno]{article}

\usepackage[letterpaper]{geometry}

\usepackage{ltexpprt}
\usepackage{hyperref}

\usepackage{color}
\usepackage{mathtools}
\usepackage[sort,numbers]{natbib}
\usepackage{bbm}

\usepackage[utf8]{inputenc} 
\usepackage[T1]{fontenc}    
\usepackage{url}            
\usepackage{booktabs}       
\usepackage{amsfonts}       
\usepackage{nicefrac}       
\usepackage{microtype}      
\usepackage{amsmath} 
\usepackage{amssymb}
\usepackage{mathrsfs}
\newtheorem{definition}{Definition}[section]

\newtheorem{remark}{Remark}[section]

\newcommand{\argmax}{\operatornamewithlimits{argmax}}
\usepackage{comment}
\allowdisplaybreaks
\usepackage[titletoc]{appendix}
\usepackage{systeme}

\newcommand{\1}{\mathbbm{1}}
\newcommand{\ve}{\varepsilon}

\newcommand{\R}{\mathbb{R} }

\usepackage{tikz}
\usepackage{pgfplots}
\usepgfplotslibrary{external}
\pgfplotsset{compat=1.10}
\usetikzlibrary{positioning}
\usetikzlibrary{patterns}
\usetikzlibrary{intersections,arrows.meta,positioning,calc}
\usepackage{comment}
\usepackage[shortlabels]{enumitem}
\usepackage{todonotes}

\usepackage{caption}
\usepackage[lined,linesnumberes,boxed]{algorithm}

\newcommand{\oP}{o_{\sss \PR}}

\newcommand{\PR}{\ensuremath{\mathbb{P}}}
\newcommand{\E}{\ensuremath{\mathbb{E}}}

\newcommand{\e}{\ensuremath{\mathrm{e}}}

\newcommand{\var}{\mathrm{Var}}
\newcommand{\cS}{\mathcal{S}}
\newcommand{\cB}{\mathcal{B}}
\newcommand{\hsig}{\hat{\sigma}}
\newcommand{\ind}[1]{\1_{\{#1\}} }
\newcommand{\sign}[1]{\mathrm{Sign} }
\usepackage{environ}
\NewEnviron{eq}{%
\begin{equation}\begin{split}
  \BODY
\end{split}\end{equation}
}
\def\sss{\scriptscriptstyle}

\newcommand{\cD}{\mathcal{D}}
\newcommand{\cF}{\mathscr{F}}
\newcommand{\cG}{\mathcal{G}}

\newcommand{\SBM}{\mathrm{CSBM}}
\newcommand{\sgn}{\mathrm{sign}}
\newcommand{\err}{\mathrm{Err}}
\newcommand{\spec}{\hsig_{\sss \mathrm{Spec}}}
\newcommand{\degree}{\hsig_{\sss \mathrm{Deg}}}
\newcommand{\best}{\hsig_{\sss \mathrm{Best}}}
\newcommand{\map}{\hsig_{\sss \mathrm{MAP}}}

\newcommand{\hX}{\hat{X}}
\newcommand{\DKL}{D_{\sss \mathrm{KL}}}
\newcommand{\true}{\sigma_0}

\renewcommand{\ln}{\log}
\numberwithin{equation}{section}

\newcommand\blfootnote[1]{%
  \begingroup
  \renewcommand\thefootnote{}\footnote{#1}%
  \addtocounter{footnote}{-1}%
  \endgroup
}  

\begin{document}

\newcommand\relatedversion{}
\renewcommand\relatedversion{\thanks{A full version of the paper can be accessed at \protect\url{arXiv:2107.06338}}} 

\title{\Large Spectral recovery of binary censored block models\relatedversion}
\author{Souvik Dhara\thanks{Department of Mathematics, Massachusetts Institute of Technology}, Julia Gaudio\thanks{Department of Industrial Engineering and Management Sciences, Northwestern University}, Elchanan Mossel$^{\dagger}$, Colin Sandon$^{\dagger}$}

\date{}

\maketitle


\fancyfoot[R]{\scriptsize{Copyright \textcopyright\ 2022 by SIAM\\
Unauthorized reproduction of this article is prohibited}}





\begin{abstract} \small\baselineskip=9pt Community detection is the problem of identifying community structure in graphs. Often the graph is modeled as a sample from the Stochastic Block Model, in which each vertex belongs to a community. The probability that two vertices are connected by an edge depends on the communities of those vertices. In this paper, we consider a model of {\em censored} community detection with two communities, where most of the data is missing as the status of only a small fraction of the potential edges is revealed. 
In this model, vertices in the same community are connected with probability $p$ while vertices in opposite communities are connected with probability $q$. The connectivity status of a given pair of vertices $\{u,v\}$ is revealed with probability $\alpha$, independently across all pairs, where $\alpha = \nicefrac{t \log(n)}{n}$. We establish the information-theoretic threshold $t_c(p,q)$, such that no algorithm succeeds in recovering the communities exactly when $t < t_c(p,q)$. We show that when $t > t_c(p,q)$, a simple spectral algorithm based on a weighted, signed adjacency matrix succeeds in recovering  the communities exactly. 

While spectral algorithms are shown to have near-optimal performance in the symmetric case, we show that they may fail in the asymmetric case where the connection probabilities inside the two  communities are allowed to be different.
In particular, we show the existence of a parameter regime where a simple two-phase algorithm succeeds but {\em any} algorithm based on the top two eigenvectors of the weighted, signed adjacency matrix fails. 

 \blfootnote{Emails: \href{mailto:sdhara@mit.edu}{sdhara@mit.edu}, \href{mailto:julia.gaudio@northwestern.edu}{julia.gaudio@northwestern.edu}, \href{mailto:elmos@mit.edu}{elmos@mit.edu}, \href{mailto:csandon@mit.edu}{csandon@mit.edu}}
\blfootnote{\emph{Acknowledgement:} 
S.D., E.M and C.S. are partially supported by Vannevar Bush Faculty Fellowship ONR-N00014-20-1-2826. E.M. and C.S. are partially supported by NSF award DMS-1737944. E.M. is partially supported by Simons Investigator award (622132) and by ARO MURI W911NF1910217. Part of this work was completed while J.G. was at the Department of Mathematics, Massachusetts Institute of Technology.}

\end{abstract}

\maketitle

\section{Introduction}
The problem of detecting community structure is an important question in the study of networks. The canonical formulation of this problem is made via the \emph{stochastic block model} (SBM), where each vertex is a member of one of $K$ communities, and vertices create an edge independently based on their latent community assignments.
The objective is to recover the latent community assignments based on the observed network. One may be interested in exact, partial, or weak recovery. When the average degree scales as $\Theta(\log n)$, such that the graph is connected with high probability, one is often  interested in recovering the community assignments exactly. Previous literature gives us precise characterization on when the exact recovery problem is efficiently solvable and when it is information theoretically impossible \cite{ABH16,MNS16,AS15,Krzakala2013}. See the survey \cite{A18} for an overview. 

Popular approaches for community detection in the stochastic block model include spectral algorithms and semidefinite programming. Given the adjacency matrix representation of the graph, spectral algorithms use properties of the eigenvalues and eigenvectors of the matrix in order to infer the communities \cite{Newman2013,Chin2015,Yun2014,Abbe2020}. 
Other approaches are based on the ``non-backtracking'' matrix of the graph \cite{Krzakala2013,Bordenave2015}. Semidefinite programming approaches \cite{Montanari2016,Hajek2016,Hajek2016b} are typically a relaxation of the maximum likelihood estimator, which is NP-hard to compute directly.

In this paper, we study the community detection problem when we only have partial knowledge about the graph. The information about the status of an edge between each pair of vertices is \emph{censored} independently, i.e., an edge between a pair of vertices can be present, absent or censored. A precise description is given below in Definition~\ref{defn:CSBM}. This model is referred to as the \emph{Censored Stochastic Block Model} (CSBM). 
Abbe, Bandeira, Bracher, and Singer \cite{ABBS14} were the first to address the exact recovery problem on CSBM when the within community edge probability $p$ and the between community edge probability $q$ satisfy $p+q =1$. They show that the Maximum Likelihood Estimator (MLE) achieves exact recovery up to the information theoretic threshold. However, 
the MLE turns out to be equivalent to the Max-Cut problem and is thus NP-hard to compute. 
For this reason, \cite{ABBS14} considered the Semidefinite Programming (SDP) relaxation which was shown to work in a certain regime. 
Later, Hajek, Wu, and Xu \cite{Hajek2015, Hajek2016} proved that, under this set up, the SDP relaxation actually works up to the information theoretic threshold. 

The impressive line of work above leaves an important question open: {\em 
what families of algorithms achieve exact recovery for the CSBM up to the information theoretic threshold}? Given that semidefinite programming is less efficient than spectral approaches, the main question of this paper is:\\

{\bf Question:}
Are spectral algorithms information theoretically optimal for exact recovery of CSBMs?\\

It is important to note here the difference between purely spectral algorithms and algorithms that combine spectral algorithms with additional clean-up procedures. The results on exact recovery for the standard block model in~\cite{ABH16,MNS16} used spectral algorithms as the first step in a two stage procedure. In the usual uncensored stochastic block model, Abbe, Fan, Wang, and Zhong \cite{Abbe2020} recently showed that the spectral algorithm is optimal, without need for a clean-up stage. It was not immediately clear whether the same would be true for the censored version that we consider, since observations are ternary (present, absent, censored) rather than binary (present, absent).

We are also interested in answering the  spectral question in a more general setup than what was previously considered. First, previous results focused on the case that $p+q=1$ which ensures that a present edge carries the same relative information as an absent edge.
We see that this condition can be avoided.
Second, we are also interested in cases where the edge probabilities within different communities are different. In our main results we show that:
\begin{itemize}
    \item[$\rhd$] spectral algorithms are optimal for the CSBM for all values of $p,q$ above the information theoretic threshold, and
    \item[$\rhd$] if the interconnection probabilities inside two communities are different, then spectral algorithms are sub-optimal in a regime of parameters.
\end{itemize}
The latter result is not a sign of a computational-statistical gap as we find a simple two stage algorithm that exactly recovers the communities in this case. 

\section{Model and main results} 
\subsection{Model description.}
We start by defining the Censored Stochastic Block Model.
\begin{definition}[Censored Stochastic Block Model (CSBM)] \label{defn:CSBM}
We have $n$ vertices with labels given by $\sigma_0 \in \{\pm 1\}^n := \mathcal{S}$. The vertices $i$ with $\true(i) = +1$ (resp.~$\true(i) = -1$) are said to be in Community~1 (resp.~Community~2). The labels are generated independently, with
\[\mathbb{P}(\sigma_0(i) = +1) = \frac{1}{2} \text{~~and~~} \mathbb{P}(\sigma_0(i) = -1) = \frac{1}{2}.\]
Two vertices $i \neq j$ are connected by an edge with probability $p_1$ if $\true(i) = \true(j) = +1$, $p_2$ if $\true(i) = \true(j) = -1$, and $q$ if $\true(i) \neq \true(j)$. Self-loops do not occur. Each edge status is revealed independently with probability $\alpha = \nicefrac{t \log n}{n}$ for a constant $t > 0$. The output is a graph with edge statuses given by $\{\texttt{present},\texttt{absent}, \texttt{censored}\}$.
We denote this censored model by $\text{CSBM}(p_1, p_2,q,\alpha)$. In the symmetric case where $p_1 = p_2 = p$, we denote the model by $\text{CSBM}(p,q,\alpha)$. For a visualization of the model, see Figure \ref{fig:sampling-diagram}. 
\begin{figure}[h]
    \centering
    \includegraphics[scale=0.6]{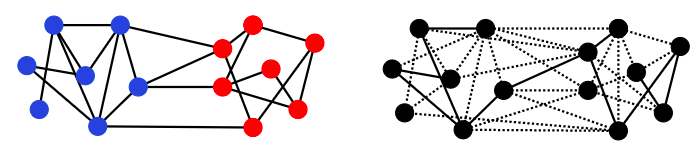}
    \caption{Blue and red vertices represent Community 1 and 2, respectively. Each pair of vertices is connected with probability $p_1$, $p_2$, or $q$ according to the communities of the vertices. The black graph is the sample that we see; vertex labels are absent, and edges may be present (solid line), absent (no line), or censored (dotted line).}
    \label{fig:sampling-diagram}
\end{figure}
\end{definition}

\noindent {\bf Objective.} We observe a graph~$G$ from $\SBM(p_1,p_2,q,t)$ with vertex labels removed (i.e., $\true$ is unknown), and edges labeled by $\{\texttt{present},\texttt{absent}, \texttt{censored}\}$. An estimator $\hsig$ is said to achieve exact recovery if 
\begin{eq}\label{def:exact-recov}
\lim_{n\to\infty}\PR(\exists s\in \{\pm 1\}: \hsig = s\true) = 1.  
\end{eq}
We say that exact recovery is possible if there exists some estimator $\hsig$ such that \eqref{def:exact-recov} holds. \\

Next, we provide formal statements of the main results, along with the key conceptual contributions.

\subsection{Optimality of the spectral algorithm in the symmetric case.} \label{sec:spec-result}
Consider the symmetric case $p_1 = p_2 = p$ and $p\neq q$.
Throughout, we assume that $p,q \in (0,1)$ and $t>0$ are fixed. We will establish that the information-theoretic threshold for exact recovery is given by:
\begin{eq}\label{threshold-defn}
 t_c(p,q):= \frac{2}{\left(\sqrt{p} - \sqrt{q} \right)^2 + \left(\sqrt{1-q} - \sqrt{1-p} \right)^2}.
\end{eq}
Our first result shows that if $t<t_c(p,q)$, then any estimator fails to achieve exact recovery with high probability. 
\begin{theorem}\label{theorem:impossibility} If $t< t_c(p,q)$, then any estimator $\hsig$ satisfies $\PR(\hsig = \true) = o(1)$.
\end{theorem}
\begin{figure}[h]
    \centering
    \includegraphics[scale = .4]{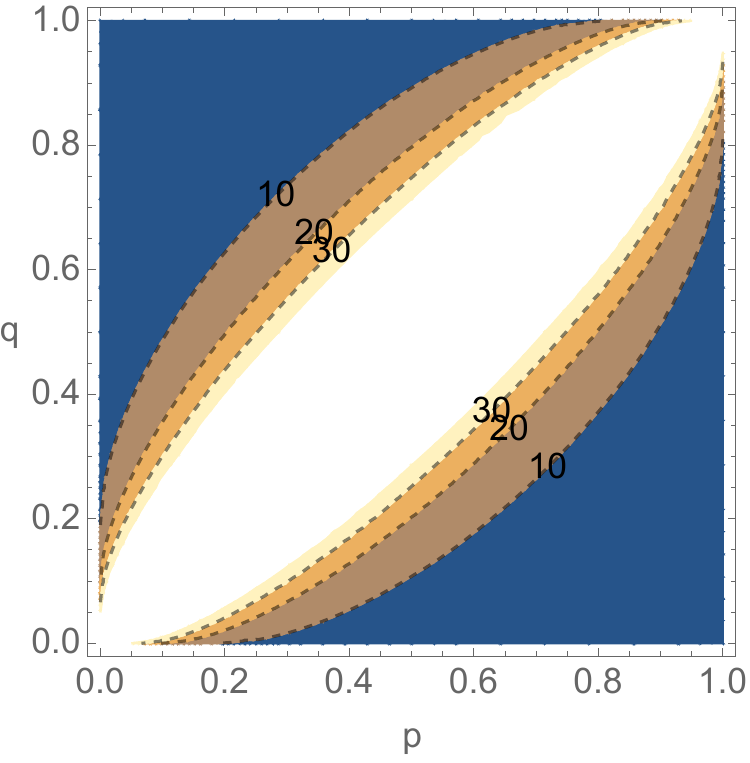}
    \caption{Contour plot for the region where exact recovery is impossible, i.e.,~$t<t_c(p,q)$ for $t= 10,20,30$. As $t$ decreases the impossibility region expands.}
    \label{fig:contour}
\end{figure}
Figure~\ref{fig:contour} illustrates the region $t< t_c(p,q)$.
Next, we describe the success of a simple spectral algorithm for $t>t_c(p,q)$. 
To this end, we need a weighted version of the signed adjacency matrix. Let 
\begin{eq}\label{eq:y-choice}
y = y(p,q) =  \frac{\log\left(\frac{1-q}{1-p} \right)}{\log\left(\frac{p}{q} \right)},
\end{eq}
and define the signed adjacency matrix as 
\begin{align*}
A_{ij} &= \begin{cases}
1 &\text{if } \{i,j\} \text{ is present}\\
-y &\text{if } \{i,j\} \text{ is absent}\\
0 &\text{if } \{i,j\} \text{ is censored}.
\end{cases}
\end{align*}
Note that $y > 0$ whenever $p\neq q$. Our next result shows that the spectral algorithm achieves exact recovery for $t>t_c(p,q)$. 
\begin{theorem}\label{theorem:exact-recovery-spectral} 
Let $u_1$ be the top eigenvector of $A$ and define $\spec = \sgn(u_1)$. 
If $t>t_c(p,q)$ and $p>q$,  
then there exists $\eta = \eta(p,q) > 0$ and $s \in \{\pm1\}$ such that, with probability $1-o(1)$,
\[s\sqrt{n} \min_{i \in [n]} \true (i)(u_1)_i \geq \eta. \]
Consequently, $\spec$ achieves exact recovery for $t> t_c(p,q)$. 
If $p<q$, all conclusions hold if we replace $A$ by $-A$. 
\end{theorem}
The estimator $\spec$ does not require additional clean-up steps. 
For the non-censored SBM, Abbe et.~al.~\cite{Abbe2020} established an analogous result using entrywise eigenvector perturbation analysis, which we will also use (see Proposition~\ref{lemma:corollary-3.1-abbe}).
The key distinction of our algorithm to that of Abbe.~et.~al.~\cite{Abbe2020}  lies in the fact that the observed graph admits a ternary encoding as opposed to a binary encoding. 
It is worthwhile to remark that the choice of $y$ is important for the spectral algorithm to succeed up to the information theoretic threshold, and in fact, \eqref{eq:y-choice} is the only choice of $y$ that works.  
Intuitively, $y$ is the ratio of  evidence provided by a present edge on a vertex's community as compared to an absent edge.
More precisely, if we compute the log-likelihood ratio for a vertex to be in one community or the other, then it turns out to be a linear function of the number of present and absent edges to each community. 
In the symmetric case, the ratio of the coefficients corresponding to present and absent edges to Community~$j$ turns out to be $y$ for both $j=1,2$. When $p+q = 1$, i.e., in the special case considered in \cite{ABBS14,Hajek2016}, we have that $y=1$, and the relative information provided by present and absent edges are equal. Therefore, Theorem~\ref{theorem:exact-recovery-spectral} shows that the spectral algorithm with $y=1$ would succeed in exact recovery in the model of \cite{ABBS14,Hajek2016}.

We conclude this section by showing that the error rate of the spectral algorithm is close to that of the best possible estimator. 
For an estimator $\hsig$, we define the error rate as
\begin{eq}\label{eq:distance-defn}
\err(\hsig) = \E\bigg[\min_{s\in \{-1,+1\}} \frac{1}{n} \sum_{i=1}^n \ind{\hsig(i) \neq s \true (i)}\bigg].
\end{eq}
To define the best estimator, we use a genie-aided approach. 
Suppose that we want to find the label of $u$. 
Now, in addition to the observed edge-labeled graph $G$, suppose a genie gives us $(\true(v))_{v\in [n]\setminus \{u\}}$. 
The genie-based estimator minimizes the probability of making an error given these observations. More precisely, the genie-based estimator $\best$ is given by 
\begin{eq}\label{eq:genie-defn}
\best(u)&:=\argmax_{r\in \{\pm 1\}} \PR(\true(u) = r \mid G, (\true(v))_{v\in [n]\setminus \{u\}}).
\end{eq}
The next result shows that the error rate of $\spec$ is within a $n^{o(1)}$ factor of the error rate of $\best$: 
\begin{theorem}\label{thm:spec-best-compared} For any fixed $t>0$, the spectral and genie-aided estimators satisfy
\[\err(\spec) = n^{o(1)}\err(\best) + O(n^{-3})= n^{- (1+o(1)) t/t_c(p,q)}+O(n^{-3}).\] 
\end{theorem}
Therefore, the expected number of misclassified vertex-labels by the spectral algorithm is at most~$n^{o(1)}$ times that of the genie estimator. 
In particular, if $t>t_c(p,q)$, then this expected number is $o(1)$ which is why exact recovery is achievable by the spectral algorithm.

\begin{remark}\normalfont If $p,q$ are unknown, then it is not difficult to estimate them. Indeed, if $E,T$ respectively denote the number of edges and triangles in the graph then $E = (1+o(1)) \frac{t n \log n}{4} (p+q)$ and $T =  (1+o(1)) \frac{t^3 \log^3 n}{8} \left(p q^2 + \frac{1}{3}p^3 \right)$ with probability tending to 1. We can use these to find consistent estimator $(\hat{p}, \hat{q})$ of $(p,q)$ and use the spectral algorithm with $\hat{y} = y(\hat{p},\hat{q})$. The conclusions of Theorem~\ref{theorem:exact-recovery-spectral} and Theorem~\ref{thm:spec-best-compared} still remain valid. 
\end{remark}

\subsection{Beating the spectral algorithm}\label{sec:beating-spectral-results}
Theorems~\ref{theorem:exact-recovery-spectral}~and~\ref{thm:spec-best-compared} strongly support  the success of the spectral algorithm. We next provide  results to show that there is room for improvement.
First, Theorem~\ref{thm:spec-best-compared} proves a $n^{o(1)}$ relative discrepancy  between the error rates of the spectral and the genie estimators, and therefore, the expected  difference between misclassified vertices under these two algorithms may grow with $n$. We prove that with an additional clean-up step, one can in fact get a $1+o(1)$ relative discrepancy  between the error rates. 
Second, we consider the case where the connection probabilities within Community 1 ($p_1$) and Community 2 ($p_2$) differ. 
In this case, we show that  the spectral algorithm  does not always work up to the information-theoretic threshold. 
Intuitively, if $p_1 \neq p_2$, then the relative information  provided by present and absent edges are different for vertices in different communities. For this reason, it is not possible to find a common choice of encoding $y$ which works well for vertices in both communities. We establish the failure of the spectral algorithm rigorously for $p_1 = 1-p_2$ and $q= \nicefrac{1}{2}$. 
On the other hand, an algorithm based on degrees and an additional clean-up step  turns out to be near-optimal for any $p_1,p_2,q$.

We start by introducing the  two-step estimator for which we need some notation.  
For a given vertex $u$ and $\sigma \in \cS$, let $D = D(\sigma, u) = (D_i(\sigma,u))_{i=1}^4$ be the degree profile where $D_1,D_2$ (resp.~$D_3,D_4$) are the number of present and absent edges to vertices in Community 1 (resp.~Community 2). Note that $D(\sigma,u)$ depends only on $G$ and $(\sigma(v))_{v\in [n]\setminus \{u\}}$.
Define $\Gamma (u,\sigma, p_1,p_2,q)$ by
\begin{eq}\label{eq:genie-lin-comb}
\Gamma (u,\sigma, p_1,p_2,q) &= D_1(\sigma,u) \log \frac{p_1}{q} + D_2(\sigma,u) \log \frac{1-p_1}{1-q} + D_3(\sigma,u) \log \frac{q}{p_2} + D_4(\sigma,u) \log \frac{1-q}{1-p_2} 
\end{eq}

\begin{definition}[Two-step estimator] 
Given an initial estimator $\hsig$, a two-step estimator $\hX(\hsig)$ is computed as follows: 
\begin{eq}\label{best:est-sym}
\hX (\hsig,u) = 
\begin{cases}
+1 &\quad \text{ if }  \Gamma (u,\hsig, p_1,p_2,q) \geq 0,\\
-1 &\quad \text{ otherwise.} 
\end{cases}
\end{eq}
\end{definition}
The function $\Gamma(\cdot)$ is designed to mimic the workings of the genie estimator. In fact, we will later show that $\best (u) = +1$ if and only if $\Gamma (u,\true, p_1,p_2,q) \geq 0$ (see Proposition~\ref{prop:genie-expression-general}). Thus the two-step estimator treats the initial estimator as a proxy of the true community assignment and then does the same procedure as the genie estimator. 
The following result states that the two-step version of the spectral estimator has a sharper error rate in the symmetric case.
\begin{theorem}\label{theorem:two-stage-spectral}
If $p_1= p_2 =p$ and $p\neq q$, then for any $t > 0$, 
\[\err(\hat{X}(\spec) ) =  (1+o(1))\err(\best) + o(1/n).\]
\end{theorem}
Next, we show that, for $p_1\neq p_2$, the two-step estimator based on degrees in the observed graph has similar strong recovery guarantees. Let $\deg(i)$ be the number of present edges incident to vertex~$i$. Define the degree-based estimator to be
\[\degree(i) = \sgn\left(\deg(i) - \frac{t \log(n)}{4}(p_1 + p_2 + 2q )\right).\] 
Let
 $c_1 = (p_1, 1-p_1, q, 1-q)$, and  $c_2 = (q,  1-q, p_2, 1-p_2)$ and define
\begin{eq}\label{defn:CH-distance}
t_c(p_1,p_2,q) = \left[1 - \frac{1}{2} \min_{x \in [0,1]} \sum_i(c_1)_i^x(c_2)_i^{1-x} \right]^{-1}.
\end{eq}
When $p_1 = p_2 = p$, then it is elementary to check that the minimum in \eqref{defn:CH-distance} is attained for $x = 1/2$ and therefore 
\begin{eq}\label{eq:critical-vals-compare}
t_c(p,p,q) = t_c(p,q),
\end{eq}
where $t_c(p,q)$ is given by \eqref{threshold-defn}.
We next provide success guarantees for $\hX(\degree)$. 
\begin{theorem}\label{theorem:two-stage-degree}
 If $p_1, p_2, q \in (0,1)$ are distinct, then for any $t > 0$,
\[\err(\hat{X}(\degree) ) =  (1+o(1))\err(\best) + o(1/n).\]
Furthermore, if $t > t_c(p_1,p_2,q)$, then $\hat{X}\left(\degree\right)$ achieves exact recovery. 
\end{theorem}
\begin{remark} \normalfont
If $p_1$, $p_2$, and $q$ are unknown, then one can estimate them as follows. First, classify each vertex according to whether its degree is above average. 
This procedure will result in an estimator $\hsig$ with at most $n^{1-\varepsilon}$ errors on average, i.e., $\err(\hsig) = O(n^{-\varepsilon})$ for some $\varepsilon > 0$. 
Then we can estimate the parameters by counting revealed edges and non-edges between and within the estimated communities. 
\end{remark}
If $p_1 = p_2$, then $\degree$ cannot achieve spectral recovery. In this sense, Theorem~\ref{theorem:two-stage-spectral} and Theorem~\ref{theorem:two-stage-degree} are complementary. We next describe the failure of the spectral algorithm for $p_1 = 1-p_2$ and $q= \nicefrac{1}{2}$. 
Let us start by defining a  version of the spectral algorithm which makes decisions based on arbitrary linear combinations of the top two eigenvectors of some encoding matrix. 
\begin{definition}
\normalfont 
Given an encoding parameter $y\in \R$, threshold $r\in \R$ and constants $\gamma_1,\gamma_2\in \R$, let $A$ be the signed adjacency matrix with entries $A_{ij} \in \{ -y, 0 , 1\}$. Let $u_1$ and $u_2$ be the two top eigenvectors of $A$. Then the spectral algorithm $\mathrm{Spectral} (y,r,\gamma_1,\gamma_2)$ outputs the estimator 
\begin{eq} \label{eq:spec-arbitrary}
\hsig (i) = \sgn(\gamma_1(u_1)_i + \gamma_2 (u_2)_i -r).
\end{eq}
\end{definition}
In other words, the spectral algorithm decides community assignments using a thresholding on a linear combination of the top two eigenvectors of some encoding matrix. Only the top two eigenvectors are included since the eigenvectors of $A$ behave like noisy versions of the eigenvectors of the rank-$2$ matrix $A^{\star}$. The following result states that even this more general algorithm fails in the \emph{antisymmetric} CSBM for $t$ sufficiently close to the recovery threshold. 

\begin{theorem}\label{theorem:antisymmetric-spectral-failure}
Let $p_1 = p = 1-p_2$ and $q = 1/2$. There exists $\delta > 0$ such that if $t<t_c(p,1-p,1/2)+ \delta$, then, for any choice of $y,r,\gamma_1,\gamma_2$,  the algorithm $\mathrm{Spectral} (y,r,\gamma_1,\gamma_2)$ fails to achieve exact recovery with probability $1-o(1)$. 
\end{theorem}
\begin{figure}
    \centering
    \includegraphics[scale=.4]{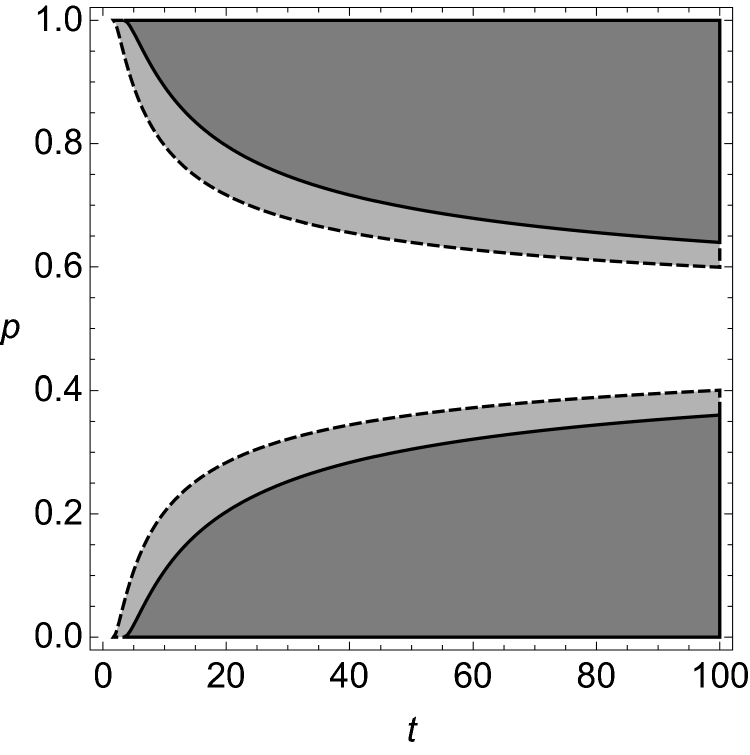}
    \caption{The estimator $\hat{X}(\degree)$ achieves exact recovery in the gray regions, corresponding to $t > t_c(p,1-p,\nicefrac{1}{2})$. No spectral algorithm can achieve exact recovery in a neighborhood of the dotted lines, represented schematically by the light gray regions.
    }
    \label{fig:spec-fail}
\end{figure}
Theorems~\ref{theorem:two-stage-degree} and~\ref{theorem:antisymmetric-spectral-failure} together show that there is a range of values of $t$ where the spectral algorithm fails in exact recovery, but $\hX(\degree)$ succeeds (see Figure~\ref{fig:spec-fail}). In other words, there is a strong separation between the spectral algorithm and the two-step procedure based on degrees. One can now see the failure of the spectral algorithm from a technical perspective.
In this regime, the top two eigenvectors of the encoding matrix $A$ are close to $Au_1^{\star}/\lambda_1^{\star}$ and $Au_2^{\star}/\lambda_2^{\star}$ respectively where $u_i^\star$ is the $i$-th largest eigenvector of the expectation matrix. If we imagine replacing $u_1$ and $u_2$ in \eqref{eq:spec-arbitrary} by their approximations, then the decision rule for the label of a vertex $u$ is given by the sign of $\sum_i z_i D_i -r_0$ for some coefficients $\{z_i\}$ and threshold $r_0$, where $D_i$ denotes the degree profile of $u$. 
Moreover, due to the choice of the encoding, the coefficients $(z_1,z_2,z_3,z_4)$ satisfy  $\frac{z_1}{z_2}= \frac{z_3}{z_4}= y$. 
The estimator that minimizes the error probability can also be shown to decide communities based on the sign of $\sum_i z_i' D_i >r'_0$, but in order to satisfy the additional condition $\frac{z_1'}{z_2'}= \frac{z_3'}{z_4'}= y$, one requires $y(p_1,q) = y(p_2,q)$, or $p_1 = p_2$. 
For this reason, the spectral algorithm is strictly less powerful than the best possible estimator and thus one cannot expect the spectral algorithm to work all the way up to the information theoretic threshold if $y(p_1,q) \neq y(p_2,q)$. Theorem~\ref{theorem:antisymmetric-spectral-failure} makes this intuition precise for $p_1=1-p_2$. 

\noindent {\bf Organization.}
The remainder of the paper is structured as follows. We start by setting up some preliminary notation in Section~\ref{sec:notation}. In Section~\ref{sec:impossibility}, we prove  impossibility for the exact recovery problem in Theorem~\ref{theorem:impossibility}. Section~\ref{sec:entrywise} is devoted to entrywise perturbation analysis of the largest eigenvector of $A$ and completing the proof of Theorem~\ref{theorem:exact-recovery-spectral}. The error analysis for the spectral algorithm and the genie estimator will be provided in Section~\ref{sec:error} and hence the proof of Theorem~\ref{thm:spec-best-compared} will be completed. In Section~\ref{sec:two-step}, we analyze the two-step estimator for a general class of initial estimators. This allows us to complete  proofs of Theorems~\ref{theorem:two-stage-spectral} and~\ref{theorem:two-stage-degree}. Finally, we conclude with a proof of Theorem~\ref{theorem:antisymmetric-spectral-failure} regarding failure of the spectral algorithm in Section~\ref{sec:failure}. We provide proof ideas at the beginnings of Sections \ref{sec:impossibility}-\ref{sec:failure}.

\section{Notation and preliminaries}\label{sec:notation}
Let $[n] = \{1, ,2, \dots, n\}$. We often use the Bachmann–Landau notation $o(1), O(1)$ etc. For two sequences $(a_n)_{n\geq 1}$ and $(b_n)_{n\geq 1}$, we write $a_n \asymp b_n$ as a shorthand for $\lim_{n\to \infty}\frac{a_n}{b_n} =1$. We write $a_n \sim b_n$ if $a_n$ and $b_n$ are asymptotically equivalent, namely $\lim_{n\to \infty}\frac{a_n}{b_n} \to c$ for some $c \neq 0$. Additionally, we write $a_n \approx b_n$ if these sequences differ by a polylogarithmic factor asymptotically, namely there exists some constant $c$ such that $a_n = O\left(b_n \log^c(n)\right)$ and $b_n = O\left(a_n \log^c(n)\right)$. For random variables $(X_n)_{n\geq 1}$, we write $X_n = \oP(1)$ as a shorthand for $X_n \to 0$ in probability.

For a vector $x \in \mathbb{R}^n$, we define $\Vert x \Vert_2 = (\sum_{i=1}^n x_i^2)^{1/2}$ and $\Vert x \Vert_{\infty} = \max_i |x_i|$. For a matrix $M \in \mathbb{R}^{n \times d}$, we use $M_{i \cdot}$ to refer to its $i$-th row, represented as a row vector. 
Given a matrix $M$, $\Vert M \Vert_2 = \max_{\Vert x \Vert_2 = 1} \Vert M x \Vert_2$ is the spectral norm, 
and $\Vert M \Vert_{2 \to \infty} = \max_i \Vert M_{i \cdot}\Vert_2$ is the matrix $2 \to \infty$ norm. 
We use the convention that $\log$ denotes natural logarithm, and write $\log^k(n)$ to mean $\left(\log(n)\right)^k$.
Finally, $\DKL(p\Vert q)$ refers to the Kullback--Leibler divergence of two Bernoulli random variables with parameters $p$ and $q$:
\[\DKL(p\Vert q) = p \log \left(\frac{p}{q}\right) + (1-p) \log \left( \frac{1-p}{1-q}\right). \]

Let $n_1(\true) = \left|\{v : \sigma_0(v) = +1 \}\right|$ and $n_2(\true) = \left|\{v : \sigma_0(v) = -1 \}\right|$.
Note that since $n_1(\true), n_2(\true)$ are marginally distributed as $\text{Bin}\left(n, \frac{1}{2}\right)$, we have that for all $\ve \in (0,1)$, 
\begin{equation}
\left|n_1 (\true) - \frac{n}{2} \right| \leq \ve n \text{~~and~~}   \left|n_2(\true) - \frac{n}{2} \right| \leq \ve n \label{eq:count-concentration}
\end{equation}
with probability at least $1 - 2\exp(-\ve^2 n/6 )$. We will often use \eqref{eq:count-concentration} with  $\ve = n^{-\frac{1}{3}}$. Additionally, let $N(u)$ be the number of vertices whose connections to $u$ are revealed: $N(u):=\{v: \{u,v\} \text{ is revealed}\}$. By \cite[Corollary 2.4]{JLR00} for $c>1$ 
\begin{equation}
\mathbb{P}\left(N(i) \geq \log^{c}(n) \right) \leq \exp\left(-\log^c(n)\right). \label{eq:neighbor-concentration} 
\end{equation}
for all sufficiently large $n$. We will use \eqref{eq:neighbor-concentration} with $c = \frac{5}{4}$ or $c=2$.

The following Poisson approximation will be used throughout. 
\begin{lemma}
\label{fact:stirling}
Let $\{W_i\}_{i=1}^{m}$ be i.i.d.~from a distribution taking three values $a,b,c$ and $\mathbb{P}(W_i = a) = \alpha p$, $\mathbb{P}(W_i = b) = \alpha (1-p)$, and $\mathbb{P}(W_i = c) = 1- \alpha$. 
Let $N_x:= \#\{i: W_i =x\}$ for $x = a,b,c$. If $m_1, m_2 = o(\log^{3/2} n)$, $m = \frac{n}{2} (1+O(\log^{-2} n ))$ and $\alpha = t \log n /n$, then 
\begin{align*}
    M(m, m_1,m_2,p):=\PR(N_a = m_1, N_b = m_2) \asymp P\Big(\frac{t p \log n }{2}; m_1\Big)P\Big(\frac{t (1-p) \log n }{2}; m_2\Big),
\end{align*}
where $P(\lambda; m)$ is the probability that a $\mathrm{Poisson}(\lambda)$ random variable takes value $m$.
\end{lemma}
The proof follows using Stirling's approximation; we provide the details in Appendix~\ref{sec:appendix-2}.
\section{Impossibility of exact recovery}\label{sec:impossibility}
In this section, we give the proof of Theorem~\ref{theorem:impossibility}. We first identify a sufficient condition under which any algorithm fails to achieve exact recovery (Proposition \ref{proposition:equiprobable}). The condition captures the idea that there are some graph instances that cannot be labeled correctly with confidence since there are multiple suitable labelings for these instances. If these graph instances are likely to occur, then the overall failure probability can be lower-bounded. Theorem \ref{theorem:impossibility} is then proven by finding a set of graphs that are difficult to label correctly.

\subsection{Sufficient condition for impossibility.}
Recall that $\cS = \{\pm 1\}^n$ is the space of possible values of $\sigma$. 
We write $g$ as a generic notation to denote the observed value of the edge-labeled graph $G$ consisting of present, absent and censored edges. 
Also, let $\cG$ be the space of all possible values of $G$. 
We write $\PR(\cdot | \sigma)$ to denote the probability distribution of $\SBM(p,q,t)$ when the community assignments are given by~$\sigma$.

Since
\begin{align*}
    \PR(\hsig \neq \sigma_0) = \sum_{g\in \cG} \PR(\hsig \neq \sigma_0 \mid G = g ) \PR(G = g),
\end{align*}
the estimator that maximizes the posterior probability $\PR(\hsig = \sigma_0 \mid G=g )$ for all $g\in \cG$ also minimizes the error probability $\PR(\hsig \neq \sigma_0)$. This  estimator is the Maximum A Posteriori (MAP) estimator. 
Thus, an optimal algorithm is devised by choosing uniformly at random  among all MAP estimates, and we denote the corresponding estimator  by $\map$. 
Next, using the fact that $\sigma_0$ is uniformly distributed on $\cS$, we must have $\mathbb{P}(\sigma_0 = \sigma | G=g) \propto \mathbb{P}( G=g| \sigma_0 = \sigma ) $. 
Then, 
\begin{align} \label{genie-MLE-equal}
\argmax_{\sigma} \mathbb{P}\left(\sigma_0 = \sigma \mid G=g\right) = \argmax_{\sigma} \mathbb{P}\left(G=g \mid \sigma_0 = \sigma \right),
\end{align}
i.e., the MAP estimator coincides with the Maximum Likelihood estimator. 

In light of this equivalence, the following result identifies a condition where the MAP estimator fails with a given probability.
\begin{proposition}\label{proposition:equiprobable}
Fix $\delta>0$. Suppose that there is  $\cG' \subset \cG $ with $\PR(G \in \cG' | \sigma_0)\geq \delta$ such that the following holds for any $g \in \cG'$: There are $k$ pairs of vertices $\{(u_i,v_i): i\in [k]\}$ with opposite community label such that if $\sigma'_0$ is obtained by swapping any one of the labels of $u_i$ and $v_i$, then  $\mathbb{P}(G = g | \sigma_0) = \mathbb{P}(G = g | \sigma_0')$. Then, conditionally on $\sigma_0$, 
the MAP estimator $\map$ fails in exact recovery with probability at least $\delta\big(1 - \frac{1}{k}\big)$.
\end{proposition}

\begin{proof}
By our underlying condition, whenever $g\in \cG'$,  the true assignment is such that swapping one of the community assignment of one of the pairs among $\{(u_i,v_i): i\in [k]\}$ results in an equiprobable assignment. In that case, the algorithm is incorrect with probability at least $1- \frac{1}{k}$, due to \eqref{genie-MLE-equal}.
 Therefore, 
 \begin{eq}
 \PR(\map \neq \sigma_0 ~\big|~ \sigma_0) \geq \PR(\map \neq \sigma_0 ~\big|~ G \in \cG', \sigma_0) \PR(G \in \cG' ~\big|~ \sigma_0) \geq \delta \Big(1- \frac{1}{k}\Big),
 \end{eq}
 and the proof follows.
\end{proof}
\begin{remark}\label{remark:even-sizes} \normalfont 
Additionally, Proposition \ref{proposition:equiprobable} holds when $\sigma_0$ is an assignment with equal community sizes, i.e., $n_1(\true) = n_2(\true)$. In this case, observe that we can treat the assignment $\sigma_0$ as though it were chosen uniformly at random from all $\sigma$ satisfying $n_1(\sigma) = n_2(\sigma)$.
 To see this, consider applying a uniformly chosen random permutation to the vertices. The inference problem does not change; however, the MAP estimator again coincides with the Maximum Likelihood estimator, and one can use an identical argument as above.
\end{remark} 

\subsection{Proof of impossibility of exact recovery}
\begin{proof}\emph{(Proof of Theorem~\ref{theorem:impossibility}).}
Throughout the proof, we condition on $\sigma_0 \in \mathcal{S}$ such that $n_1(\true), n_2(\true) = (1+O(n^{-\nicefrac{1}{3}}))\frac{n}{2}$, which occurs with probability at least $1-2\exp(-n^{\nicefrac{1}{3}}/6)$ by \eqref{eq:count-concentration}. For convenience, we write $n_1,n_2$ instead of  $n_1(\true), n_2(\true)$. 
We will show that with high probability, there exist $k=\omega(1)$ pairs of vertices $\{(u_i,v_i):i\in [k]\}$ with opposite communities such that swapping their labels results in an equiprobable graph instance. By Proposition~\ref{proposition:equiprobable}, this would show that exact recovery fails with probability $1 - o(1)$ for any algorithm.

For $j=1,2$, let $S_j$ be sets of $\lfloor 2n_j/\ln^2(n)\rfloor \asymp \lfloor n/\ln^2(n)\rfloor $ randomly selected vertices from Community 1 and Community 2, respectively. 
Let $S = S_1 \cup S_2$. 
Next, let $S'$ be the set of all vertices in $S$ whose connections to all other vertices in $S$ are censored. We claim that $|S'|> 3n/2\ln^2(n)$ with probability $1-o(1)$. To see this, observe that the expected number of revealed connections between vertices in $S$ is at most $\alpha(2n/\ln^2(n))^2=4tn/\ln^3(n)$, so with high probability there are fewer than $n/4\ln^2(n)$ such connections. Therefore with high probability there are fewer than $\lfloor n/2\ln^2(n) \rfloor$ vertices with at least one neighbor in $S$, from which the claim follows.

Now, let $m=\lfloor \sqrt{pq} t\ln(n)/2\rfloor$, $m'=\lfloor\sqrt{(1-p)(1-q)}t\ln(n)/2\rfloor$, and $m''=m+m'$. Let $\overline{p}_1$ be the probability that a vertex $v$ in Community 1 has exactly $m$ present edges and $m'$ absent edges to vertices in each community, conditioned on $v\in S'$. 
By Lemma~\ref{fact:stirling}, we have 
\begin{align*}
   \overline{p}_1&= M(n_1 - \lfloor 2n_1/\ln^2(n)\rfloor, m,m',p) \times M(n_2 - \lfloor 2n_2/\ln^2(n)\rfloor, m,m',q) \\
   & \asymp \e^{-t\log n}\frac{(\frac{t^2 pq \log^2 n}{4})^{m}(\frac{t^2 (1-p)(1-q) \log^2 n}{4})^{m'}}{(m!)^2(m'!)^2}  \\
   &\asymp n^{-t} \frac{(\frac{t^2 pq \log^2 n}{4})^{m}(\frac{t^2 (1-p)(1-q) \log^2 n}{4})^{m'}}{(2\pi)^2 \e^{-2m} m^{2m+1} \e^{-2m'} (m')^{2m'+1}} \\
   & = \frac{n^{-t}}{4\pi^2 m m'} \e^{2m+2m'} \bigg(\frac{t^2 pq \log^2 n}{4m^2}\bigg)^{m}\bigg(\frac{t^2 (1-p)(1-q) \log^2 n}{4m'^2}\bigg)^{m'} \\
   &\approx n^{-t }n^{\sqrt{pq}t+\sqrt{(1-p)(1-q)}t} \\ &=n^{-\left(\sqrt{p}-\sqrt{q}\right)^2t/2-\left(\sqrt{1-p}-\sqrt{1-q}\right)^2t/2}.
\end{align*}
This implies that $\overline{p}_1=\omega(\log^2(n)/n)$ because $[ (\sqrt{p}-\sqrt{q})^2+(\sqrt{1-p}-\sqrt{1-q})^2]t/2 <1$ since $t<t_c(p,q)$. Repeating the calculation for the case that $v$ is in Community 2, we conclude that the probability $\overline{p}_2$ that a given vertex $v \in S'$ has exactly $m$ present edges and $m'$ absent edges to vertices in each community is $\omega\left(\ln^2(n)/n\right)$. 

For $v \in S'$, let $Y(v)$ be the indicator that $v$ has exactly $m$ present edges and $m'$ absent edges to each community. Note that the random variables in the set $\{Y(v)\}_{v \in S'}$ are mutually independent conditionally on $S'$.
Finally, observe that if $u \in S' \cap S_1$ and $v \in S' \cap S_2$ satisfy $Y(u) = Y(v) = 1$, then switching the community labels of $u$ and $v$ results in an equiprobable outcome.

Let 
\[Y_1 = \sum_{v \in S' \cap S_1} Y(v) ~~~\text{and}~~~ Y_2 = \sum_{v \in S' \cap S_2} Y(v).\]
It suffices to show that there is a function $f(n) = \omega(1)$ such that $Y_1, Y_2 \geq f(n)$ with probability $1-o(1)$. 
We prove the claim for $Y_1$, and the proof for $Y_2$ follows similary. 
Observe that conditioning on $|S' \cap S_1|$,
\[\mathbb{E}[Y_1 ~\big|~|S' \cap S_1|] = |S' \cap S_1| \cdot \overline{p}_1.\] 
Fix $\ve > 0$. By Chebyshev's inequality, 
\begin{align*}
\mathbb{P}\left(Y_1 \leq (1- \ve) |S' \cap S_1| \cdot \overline{p}_1 ~\big|~ |S' \cap S_1| = s \right) &\leq \frac{\var\left(Y_1~\big|~ |S' \cap S_1| = s \right)}{\ve^2 s^2  \overline{p}_1^2}\leq \frac{s \overline{p}_1(1 - \overline{p}_1)}{\ve^2 s^2 \overline{p}_1^2} \leq \frac{1}{\ve^2 s \overline{p}_1}.
\end{align*}
Recall that $|S'| > \frac{3n}{2 \log^2(n)}$ with probability $1-o(1)$. Therefore, using $|S_2| = \lfloor n/\log^2n\rfloor$, we have $|S' \cap S_1| > \frac{n}{2 \log^2(n)}$ with probability $1-o(1)$. We conclude that
\begin{align*}
\mathbb{P}\left(Y_1 \leq (1-\ve) \frac{n}{2 \log^2(n)}\overline{p}_1 \right) &\leq \mathbb{P}\bigg(Y_1 \leq (1- \ve) |S' \cap S_1| \cdot \overline{p}_1 ~\Big|~ |S' \cap S_1| > \frac{n}{2 \log^2(n)} \bigg) + o(1)\\
&\leq \frac{2\log^2(n)}{\ve^2 \overline{p}_1 n} + o(1).
\end{align*}
Recalling that $\overline{p}_1= \omega(\log^2(n)/n)$, we have shown that there is a function $f(n) = \omega(1)$ such that $\mathbb{P}\left(Y_1 \leq f(n)\right) = o(1)$. Similarly, using $\overline{p}_2= \omega(\log^2(n)/n)$, it holds that $\mathbb{P}\left(Y_2 \leq f(n)\right) = o(1)$. Applying Proposition \ref{proposition:equiprobable} with $\delta = 1 - o(1)$ and $k = \left(f(n)\right)^2$ completes the proof. 
\end{proof}

\begin{remark} \normalfont
We could generalize this to an argument that recovery is impossible whenever $t<t_c(p_1,p_2,q)$ by arguing that there will be vertices with degree profiles of \[\left(\lfloor p_1^xq^{1-x}t \log(n)/2\rfloor ,\lfloor (1-p_1)^x(1-q)^{1-x}t\log(n)/2\rfloor,\lfloor p_2^{1-x}q^xt\log(n)/2\rfloor,\lfloor(1-p_2)^{1-x}(1-q)^x t\log(n)/2\rfloor \right)\]
in both communities, where $x$ takes on the value used in the computation of $t_c(p_1,p_2,q)$. The argument that this holds is largely analogous to that in the proof above, although one needs to use the fact that the criterion used to choose $x$ implies that vertices in each community are approximately equally likely to have this community profile and then bound both probabilities by bounding the weighted geometric means of their obvious formulations with weights $x$ and $1-x$.
\end{remark}

\section{Analysis of the spectral algorithm}\label{sec:entrywise}
Recall the signed adjacency matrix from Section~\ref{sec:spec-result}. 
The key to establishing Theorem \ref{theorem:exact-recovery-spectral} is the method of entrywise eigenvector analysis of \cite{Abbe2020}. Let us denote  $A^{\star} = \mathbb{E}[A \mid \true]$, and let $(u_k^{\star}, \lambda_k^{\star})$ denote the $k$-th eigenvector-eigenvalue pair of $A^{\star}$, where $(\lambda_k^\star)_k$ are arranged in non-decreasing order. Abbe et.~al.~\cite{Abbe2020} show that, under a set of general conditions, $u_k \approx Au_k^{\star}/\lambda_k^{\star}$ in the $\ell_{\infty}$-norm. Results of this kind were also derived recently by Lei \cite{Lei2019}.
Thus, if we can show that the signs of $Au_1^{\star}/\lambda_1^{\star}$ recover the communities with high probability (up to a global flip), and the magnitude of its entries are bounded away from zero, then the signs of $u_1$ also recover the communities with high probability. More precisely, using the methods of  \cite[Theorem 2.1]{Abbe2020}, we will establish the following result: 
\begin{proposition}\label{lemma:corollary-3.1-abbe}
With probability $1 - O\left(n^{-3}\right)$ we have
\[\min_{s \in \{\pm 1\}} \left \Vert u_k - s\frac{A u_k^{\star}}{\lambda_k^{\star}} \right\Vert_{\infty} \leq \frac{C}{\sqrt{n} \log \log n} \]
for $k \in \{1,2\}$, where $C = C(p,q,t)$ is a constant depending only on $p$, $q$, and $t$.
\end{proposition}
In Section \ref{sec:prior-work-entrywise}, we will provide a main result of \cite{Abbe2020}, specialized to our setting. In Section \ref{sec:eigenvector-approximation}, we prove Proposition \ref{lemma:corollary-3.1-abbe}. With this result in hand, we provide the proof of Theorem \ref{theorem:exact-recovery-spectral} in Section \ref{sec:success-spectral-theorem}.

\subsection{Prior work on entrywise eigenvector analysis.}\label{sec:prior-work-entrywise}
We start by reproducing \cite[Theorem 2.1]{Abbe2020}, specialized to the case where $A^{\star}$ is a rank-2 matrix and we wish to approximate a single eigenvector.  
\begin{theorem}\emph{(\cite[Theorem 2.1]{Abbe2020}).}
\label{theorem:theorem-2.1-abbe}
Let $A$ be a symmetric random matrix and  $A^\star = \E[A]$. Suppose that the following conditions hold with some $\gamma \in \mathbb{R}$ and $\varphi: \mathbb{R} \to \mathbb{R}$:
\begin{enumerate}[(i)]
    \item (Incoherence)\label{assumption:assumption-1} \ $\Vert A^{\star} \Vert_{2 \to \infty} \leq \gamma \Delta^{\star}$, where $\Delta^{\star} = (\lambda_1^{\star} - \lambda_2^{\star}) \wedge |\lambda_1^{\star}|$ and $\gamma>0$.
    \item \label{assumption:assumption-2} (Row- and column-wise independence) \ 
    For any $m \in [n]$, the entries in the $m$-th row and column of $A$ are independent with others. 
    \item \label{assumption:assumption-4} (Row concentration) \ 
    Suppose $\varphi(x)$ is continuous and non-decreasing in $\mathbb{R}_+$ with $\varphi(0) = 0$, $\varphi(x)/x$ is non-increasing in $\mathbb{R}_+$, and $\delta_1 \in (0,1)$. For any $m \in [n]$ and $w \in \mathbb{R}^n$,
    \[\mathbb{P}\left(\left|(A - A^{\star})_{m  \cdot} w\right| \leq \Delta^{\star} \Vert w \Vert_{\infty} \varphi\left(\frac{\Vert w \Vert_2}{\sqrt{n} \Vert w \Vert_{\infty}} \right)\right) \geq 1 - \frac{\delta_1}{n}. \] 
    \item \label{assumption:assumption-3} (Spectral norm concentration) \ Let $\kappa = |\lambda_1^{\star}|/\Delta^{\star}$. Suppose  $32 \kappa \max\{\gamma, \varphi(\gamma)\} \leq 1$ and for some $\delta_0 \in (0,1)$, 
    \[\mathbb{P}\left(\Vert A - A^{\star}\Vert_2 \leq \gamma \Delta^{\star}  \right) \geq 1 - \delta_0. \]
    \end{enumerate}
    Then with probability at least $1 - \delta_0 - 2 \delta_1$,
    \[\min_{s \in \{\pm1\}} \left \Vert u_k - s\frac{Au_k^{\star}}{\lambda_k^{\star}}\right \Vert_{\infty} \lesssim \kappa\left(\kappa + \varphi(1)\right)\left(\gamma + \varphi(\gamma)\right) \Vert u_1^{\star}\Vert_{\infty} + \gamma \frac{\Vert A^{\star}\Vert_{2 \to \infty}}{\Delta^{\star}}\]
    for $k \in \{1,2\}$, where $\lesssim$ hides an absolute constant.
\end{theorem}
Next, we state the following two lemmas to verify the final two conditions in Theorem~\ref{theorem:theorem-2.1-abbe}. 
The first is similar to \cite[Theorem 9]{Hajek2016}.
\begin{lemma}\label{lemma:supporting-1}
There exists $c_1 = c_1(p,q,t)>0$ such that
\[\mathbb{P}\left(\Vert A - A^{\star} \Vert_2 \geq c_1 \sqrt{\log(n)}\right) \leq n^{-3}.\]
\end{lemma}
The following lemma is similar to \cite[Lemma 7]{Abbe2020}.
\begin{lemma}\label{lemma:supporting-2}
Let $w \in \mathbb{R}^n$ be a fixed vector, $\{X_i\}_{i=1}^n$ be independent random variables where $\mathbb{P}(X_i = 1) = p_i$, $\mathbb{P}(X_i = -y) = q_i$, and $\mathbb{P}(X_i = 0) = 1-p_i -q_i$. Let $\beta \geq 0$. Then
\[\mathbb{P}\bigg( \bigg|\sum_{i=1}^n w_i\left(X_i - \mathbb{E}[X_i] \right) \bigg| \geq \frac{\max\{1,y\}(2+ \beta)n}{1 \vee \log\left(\frac{\sqrt{n} \Vert w \Vert_{\infty}}{\Vert w \Vert_2} \right)} \Vert w \Vert_{\infty} \max_i \{p_i + q_i\} \bigg) \leq 2 \exp\left(-\beta n \max_i\{p_i + q_i\} \right). \]
\end{lemma}
The proofs of the above lemmas will be provided in Appendix~\ref{sec:appendix-supporting}. 

\subsection{Proof of eigenvector approximation result.}\label{sec:eigenvector-approximation}
We start by determining the eigenvalues and eigenvectors of $A^\star =\mathbb{E}[A \mid \true]$.
\begin{lemma}\label{lem:evals-Astar}
If $n_1(\sigma_0), n_2(\sigma_0) \geq 1$, then $A^\star$ has rank 2. If $p>q$, then with probability at least $1 - 2\exp\left(-n^{\nicefrac{1}{3}}/6\right)$, the eigenvalues of $A^\star$ are given by
\begin{align*}
\lambda_1^{\star} = \frac{(1+o(1))t \log n}{2 \log \big(\frac{p}{q} \big)} \left(\DKL\left(p\Vert q \right) + \DKL\left(q \Vert p \right) \right), \quad 
\lambda_2^{\star} = \frac{(1+o(1))t \log n}{2 \log \big(\frac{p}{q} \big)} \left(\DKL\left(p\Vert q \right) - \DKL\left(q \Vert p \right) \right)
\end{align*}
and the corresponding eigenvectors $u_1^\star$ and $u_2^\star$ are respectively given by  $(u_1^{\star})_i = \frac{1+o(1)}{\sqrt{n}}$ if $\sigma_0(i) = +1$,  $(u_1^{\star})_i = -\frac{1+o(1)}{\sqrt{n}}$ if $\sigma_0(i) = -1$, and the other eigenvector has $(u_2^{\star})_i =\frac{1}{\sqrt{n}}$ for all $i$. 
\end{lemma}
\begin{proof}
Recall the edges are revealed independently with probability $\alpha = t \log n / n$. Thus, for $i,j$ such that $\sigma_0(i) = \sigma_0(j)$, 
\begin{align}\label{eq:A-star-1}
A^\star_{ij} = \alpha \bigg(p  - \frac{\log \big(\frac{1-q}{1-p}\big)}{\log \big(\frac{p}{q}\big)}(1-p)\bigg) =\frac{t \log n}{n\log \big(\frac{p}{q} \big)} \DKL(p\Vert q ),
\end{align}
and similarly, for $\sigma_0(i) \neq \sigma_0(j)$, 
\begin{align}\label{eq:A-star-2}
    A^\star_{ij} = \alpha \bigg(q  - \frac{\log \big(\frac{1-q}{1-p}\big)}{\log \big(\frac{p}{q}\big)}(1-q)\bigg) =-\frac{t \log n}{n\log \big(\frac{p}{q} \big)} \DKL(q\Vert p ).
\end{align}
If $n_1(\true) = n_2(\true) = \frac{n}{2}$, it follows that 
\begin{align*}
\lambda_1^{\star} = \frac{t \log n}{2 \log \big(\frac{p}{q} \big)} \left(\DKL\left(p\Vert q \right) + \DKL\left(q \Vert p \right) \right), \quad 
\lambda_2^{\star} = \frac{t \log n}{2 \log \big(\frac{p}{q} \big)} \left(\DKL\left(p\Vert q \right) - \DKL\left(q \Vert p \right) \right)
\end{align*}
By \eqref{eq:count-concentration} with e.g.~$\ve =n^{-\frac{1}{3}}$, we have that $n_1(\true),n_2(\true) = (1+o(1))n/2$ with probability at least $1 - 2\exp(-n^{\nicefrac{1}{3}}/6)$. We now consider what happens to the eigenvalues of $A^{\star}$ when $n_1(\true),n_2(\true)$ are perturbed.

More generally, let $Z(a_1, a_2, b_1, b_2)$ denote an $n\times n$ block matrix with blocks of size $b_1 n$ and $b_2 n$, where the diagonal blocks take value $a_1$ and the off-diagonal blocks take value $a_2$. 
Let $\lambda$ be an eigenvalue of $Z(a_1, a_2, b_1, b_2)$ and let $\lambda'$ be the corresponding eigenvalue of $Z(a_1, a_2, b_1', b_2')$. Let $E = Z(a_1, a_2, b_1', b_2') - Z(a_1, a_2, b_1, b_2)$. By Weyl's inequality,
\begin{align*}
|\lambda - \lambda'| &\leq \Vert E\Vert_2\leq \Vert E \Vert_F \leq \sqrt{(a_1 - a_2)^2 \left(|b_1 - b_1'| + |b_2 - b_2'| \right)n } = |a_1 - a_2| \sqrt{2 |b_1 - b_1'| n}. 
\end{align*}
In particular, if $\lambda = \Theta(\log(n))$, $a_1, a_2 = \Theta(\frac{\log(n)}{n})$, and $|b_1' - b_1| = o(n)$, then $\lambda' = (1+o(1))\lambda$. We conclude that when $n_1(\true), n_2(\true) = (1+o(1))n/2$, the eigenvalues of $A^{\star}$ are the same as the even communities case up to a $1+o(1)$ factor.

Regarding the eigenvector $u_1^{\star}$, its entries are given by $\pm \frac{1}{\sqrt{n}}$ depending on community membership, in the case of $n_1(\true) = n_2(\true) = n/2$. When $n_1(\true),n_2(\true) = (1+o(1))n/2$, then determining the entries of the eigenvector $u_{1,i}^{\star} \in \{x_1, x_2\}$ requires solving a system of the form
\begin{align*}
\begin{cases}
n_1(\true) a_1 x_1 + n_2(\true) a_2 x_2 = \lambda_1 x_1\\
n_1(\true) a_2 x_1 + n_2(\true) a_1 x_2 =\lambda_1 x_2.
\end{cases}
\end{align*}
Since $n_1(\true), n_2(\true) = (1+o(1)) n/2$ and $\lambda_1 = (1+o(1))\lambda_1^{\star}$, the coefficients of the system are perturbed by a factor $1+o(1)$ relative to the equal-sized communities case. Therefore, the eigenvector entries are also perturbed within a $1+o(1)$ factor.
\end{proof}

\begin{proof}\emph{(Proof of Proposition~\ref{lemma:corollary-3.1-abbe}).}
We will apply Theorem \ref{theorem:theorem-2.1-abbe} by verifying its conditions. In this proof, we avoid writing the $(1+o(1))$ terms for $\lambda_1^\star, \lambda_2^\star$ in Lemma~\ref{lem:evals-Astar} since that does not affect the asymptotic computations. We give the proof first for the case $p > q$. By Lemma~\ref{lem:evals-Astar}, it follows that \[\Delta^{\star} = \lambda_1^{\star} - \lambda_2^{\star} = \frac{\DKL\left(q \Vert p \right)}{\log \big(\frac{p}{q} \big)} t \log n ~~~~~\text{and}~~~~~ \kappa = \frac{\DKL\left(p\Vert q \right) + \DKL\left(q \Vert p \right)}{2\DKL\left(q \Vert p \right)}.\]
Let 
\[\gamma = \frac{c_1\log\big(\frac{p}{q} \big)}{\DKL(q \Vert p)t\sqrt{\log(n)}},\]
where $c_1$ is the value from Lemma \ref{lemma:supporting-1}. Let 
\[\varphi(x) = \frac{ \max\{1,y\} \left( 2t+4 \right) \log\big(\frac{p}{q} \big)}{ t\DKL(q \Vert p)}\bigg(1 \vee \log\Big(\frac{1}{x}\Big) \bigg)^{-1}.\]
To check Condition~\ref{assumption:assumption-1}, recall that $\|A^\star\|_{2\to \infty} = \max_i\|A^\star_{i\cdot}\|_2$.
Thus, by \eqref{eq:A-star-1} and \eqref{eq:A-star-2},  \[\Vert A^{\star}\Vert_{2 \to \infty} \leq \frac{t \log n}{\sqrt{n}\log \big(\frac{p}{q} \big)} \max\{\DKL(p\Vert q ), \DKL(q\Vert p ) \}.\] 
On the other hand, $\gamma \Delta^{\star} = c_1 \sqrt{\log(n)}$. Condition~\ref{assumption:assumption-1} therefore holds for $n$ large enough. Condition~\ref{assumption:assumption-2} holds since the entries $\{A_{ij}: i \leq j\}$ are independent conditioned on the communities. The first requirement of Condition \ref{assumption:assumption-3} is satisfied for $n$ sufficiently large since $\gamma \to 0$ as $n \to \infty$,  $\lim_{x \to 0} \varphi(x) = 0$ and $\kappa = \Theta(1)$. The second requirement is satisfied by Lemma \ref{lemma:supporting-1}, with $\delta_0 = n^{-3}$. Finally, to verify Condition \ref{assumption:assumption-4}, fix $m$, and apply Lemma \ref{lemma:supporting-2} with $X_i = A_{mi}$, and setting $\beta = \frac{4}{t}$. Note that $p_i$ equals $\alpha p$ or $\alpha q$ depending on whether $\sigma_0(i) = \sigma_0(m)$ or not, and let $q_i = \alpha -p_i $
which equals to either $\alpha (1-p)$ or $\alpha (1-q)$. Thus, $\max_i(p_i + q_i) = \alpha$. Then
\[\mathbb{P}\bigg(\left|(A - A^{\star})_m \cdot w \right| \leq \frac{\max\{1,y\}(2t+4) \log(n)}{1 \vee \log\left(\frac{\sqrt{n}\Vert w \Vert_{\infty}}{\Vert w \Vert_2} \right)} \Vert w \Vert_{\infty} \bigg) \geq 1 - 2n^{-4}. \]
Observing
\begin{align*}
\frac{\max\{1,y\}(2t+4) \log(n)}{1 \vee \log\left(\frac{\sqrt{n}\Vert w \Vert_{\infty}}{\Vert w \Vert_2} \right)} \Vert w \Vert_{\infty} &= \Delta^{\star} \Vert w \Vert_{\infty} \varphi\left(\frac{\Vert w \Vert_2}{\sqrt{n} \Vert w \Vert_{\infty}} \right),
\end{align*}
Condition~\ref{assumption:assumption-4} holds with $\delta_1 = 2n^{-3}$.
Applying Theorem \ref{theorem:theorem-2.1-abbe}, with probability at least $1 - 5n^{-3}$,
\[\min_{s \in \{\pm1\}}\left \Vert u_1 - s \frac{A u_1^{\star}}{\lambda_2^{\star}} \right\Vert_{\infty} \leq \frac{C}{\sqrt{n} \log \log n}, \]
where $C$ depends only on $p$, $q$, and $t$.

In the case $p < q$, we replace $A$ by $-A$. Replacing all instances of $\log(\nicefrac{p}{q})$ by $\log(\nicefrac{q}{p})$, the proof follows verbatim.
\end{proof}

\subsection{Success of the spectral algorithm.}\label{sec:success-spectral-theorem}
We will use the following concentration result which can be proved analogously to the Chernoff bound. The proof of this lemma is provided in Appendix~\ref{sec:appendix-supporting}. 
\begin{lemma}\label{lemma:supporting-3}
Let $p,q,t$ be constants such that $p > q$ and $\alpha = t \log n/n$.
Suppose $n_1,n_2 = (1+o(1)) \frac{n}{2}$. 
Let $\{W_i\}_{i=1}^{n_1}$ be i.i.d.~where $\mathbb{P}(W_i = 1) = \alpha p$, $\mathbb{P}(W_i = -y) = \alpha (1-p)$, and $\mathbb{P}(W_i = 0) = 1- \alpha$ and $y$ is given by \eqref{eq:y-choice}. Let $\{Z_i\}_{i=1}^{n_2}$ be i.i.d.~where $\mathbb{P}(Z_i = 1) = \alpha q$, $\mathbb{P}(Z_i = -y) = \alpha (1-q)$, and $\mathbb{P}(Z_i = 0) = 1- \alpha$, independent of the $W_i$'s. For any $\ve \geq 0$,  we have the following:
\begin{align*}
&\log \mathbb{P}\bigg(\sum_{i=1}^{n_1} W_i - \sum_{i=1}^{n_2} Z_i \leq \ve \log(n) \bigg)  \\
&\quad \leq \log(n) \bigg[-\lambda \ve  - 
\frac{t}{2} \bigg(\left(\sqrt{p} - \sqrt{q} \right)^2 + \left(\sqrt{1-q} - \sqrt{1-p} \right)^2 \bigg) +o(1) \bigg].
\end{align*}
\end{lemma}

\begin{proof}\emph{(Proof of Theorem~\ref{theorem:exact-recovery-spectral}).}
Consider the case $p > q$. Let \[\mathcal{C}_n:=\left\{\left| n_1(\true) - \frac{n}{2}\right| \leq n^{\frac{2}{3}}, \left| n_2(\true) - \frac{n}{2}\right| \leq n^{\frac{2}{3}} \right\}.\] 
Note that $\mathbb{P}(\mathcal{C}_n^c) = o(1)$, and indeed is much smaller. Therefore, it is sufficient to analyze events conditioned on $\mathcal{C}_n$.

For any labeling $\sigma$, define  $J_i(\sigma):= \{j \in [n] \setminus \{i\}: \sigma(j) = +1\}$. Let $s \in \{\pm 1\}$ be such that $\| u_1 - s A u_1^{\star}/\lambda_1^{\star}\|_\infty$ is minimized. 
Then, by Proposition~\ref{lemma:corollary-3.1-abbe}, with probability $1 - o(1)$, 
\begin{eq}\label{simpl-prod-eig-null}
\sqrt{n} \min_{i \in [n]} s \true(i) (u_1)_i \geq \sqrt{n} \min_{i \in [n]} \true(i) \frac{(A u_1^{\star})_i}{\lambda_1^{\star}} - C \left(\log \log n \right)^{-1},
\end{eq}
where we have used $s^2 = 1$. We now show that $\true(i)\frac{(A u_1^{\star})_i}{\lambda_1^{\star}}$ is bounded away from zero. By Lemma~\ref{lem:evals-Astar}, $(u_1^{\star})_i$ takes values $\frac{(1+o(1))}{\sqrt{n}}$ or $-\frac{(1+o(1))}{\sqrt{n}}$ depending on $\sigma_0(i)$, conditioned on $\mathcal{C}_n$. Thus, for each $i$, 
\begin{align*}
\sigma_0(i)\frac{(A u_1^{\star})_i}{\lambda_1^{\star}} &= \frac{2 (1+o(1))\log \big(\frac{p}{q} \big)}{t \left(\DKL(p\Vert q) + \DKL(q\Vert p) \right)\sqrt{n} \log(n)}\bigg(\sum_{j \in J_i(\true)} A_{ij} - \sum_{j \notin J_i(\true)} A_{ij} \bigg).
\end{align*}
Observe that for $\ve > 0$,
\begin{eq}\label{eq:spec-disc-events}
&\mathbb{P}\bigg(\sqrt{n} \sigma_0(i)\frac{(A u_1^{\star})_i}{\lambda_1^{\star}} \leq \frac{2 \ve \log \big(\frac{p}{q} \big)}{t \left(\DKL(p\Vert q) + \DKL(q\Vert p) \right)} ~\Big|~ \mathcal{C}_n \bigg) \\
&= \mathbb{P}\bigg(\sum_{j \in J_i(\true)} A_{ij} - \sum_{j \notin J_i(\true)} A_{ij} \leq \ve \log(n)  ~\Big|~ \mathcal{C}_n \bigg).
\end{eq}
\normalsize
By Lemma \ref{lemma:supporting-3}, if we have $t >t_c(p,q)$ with $t_c(p,q)$ given by \eqref{threshold-defn}, then there exists $\ve > 0$ so that
\[\mathbb{P}\bigg(\sum_{j \in J_i(\true)} A_{ij} - \sum_{j \notin J_i(\true)} A_{ij} \leq \ve \log(n)  ~\Big|~ \mathcal{C}_n \bigg) = o \Big(\frac{1}{n}\Big).\]
By a union bound and using \eqref{simpl-prod-eig-null}, we conclude that there exists $\eta = \eta(p,q)>0$ such that with probability $1 - o(1)$, 
\begin{eq}
\sqrt{n} \min_{i \in [n]} s \true(i) (u_1)_i \geq \eta.
\end{eq}

In the case $p < q$, we replace $A$ by $-A$, and the proof follows verbatim.
\end{proof}

\begin{remark}\normalfont
Note that Theorem \ref{theorem:impossibility} applies to the model considered by Hajek et.~al.~\cite{Hajek2016}, where $n_1 = n_2 = \frac{n}{2}$ (due to Remark \ref{remark:even-sizes}). Additionally, the success of the spectral algorithm (Theorem \ref{theorem:exact-recovery-spectral}) holds for this model. Therefore, Theorems~\ref{theorem:impossibility}~and~\ref{theorem:exact-recovery-spectral} are directly comparable to \cite{Hajek2016} in the special case $p+q = 1$.
\end{remark}

\section{Asymptotic error of the genie estimator} \label{sec:error}
In this section, we complete the proof of Theorem~\ref{thm:spec-best-compared}. 
In order to analyze the genie estimator, we first use the fact that the prior on $\sigma_0(u)$ is uniform, so that
\begin{equation*}
\best(u) = \argmax_{r\in \{\pm 1\}} \PR( G\mid \true(u) = r, (\true(v))_{v\in [n]\setminus \{u\}}).
\end{equation*}
In other words, the genie estimator may be interpreted as a Maximum Likelihood Estimator. Recall $\Gamma (u,\sigma, p_1,p_2,q)$ from \eqref{eq:genie-lin-comb} and the notation $D = D(\sigma, u) = (D_i(\sigma,u))_{i=1}^4$ for the degree profile from Section~\ref{sec:beating-spectral-results}.
We first derive an expression for the genie estimator for general CSBM with possibly arbitrary choices of $p_1,p_2$. This result will also be useful in the next section.
\begin{proposition} \label{prop:genie-expression-general}
For any $p_1,p_2, q \in (0,1)$, we have for all $u\in [n]$
\begin{eq}\label{eq:genie-expression-general}
\best(u) = 
\begin{cases}
+1 & \text{ if }\Gamma (u,\sigma_0, p_1,p_2,q) \geq 0 \\
-1 & \text{ otherwise. }
\end{cases}
\end{eq}
\end{proposition}
The proof is provided in Appendix \ref{appendix-3}. In other words, the genie estimator decides community assignments based on the sign of a linear combination of the degree profiles. 
For $p_1 = p_2 = p$, it is not difficult to see that if $y$ is given by \eqref{eq:y-choice}, then we get the following cleaner expression in terms of the signed adjacency matrix: 
\begin{eq}
\Gamma (u,\sigma, p,p,q) &= \bigg(\sum_{v\in J_u(\sigma) }A_{ij} -  \sum_{v\notin J_u(\sigma)}A_{ij}\bigg)\log \frac{p}{q}, 
\end{eq}
where  $J_u(\sigma):= \{v \in [n] \setminus \{u\}: \sigma(v) = +1\}$. Thus, we see that in the symmetric case, the genie estimator decides the communities based on the sign of $\sum_{v\in J_u(\sigma) }A_{ij} -  \sum_{v\notin J_u(\sigma)}A_{ij}$. From the proof of Theorem \ref{theorem:exact-recovery-spectral}, we see that the spectral algorithm recovers $\true(u)$ successfully if  $\sum_{v\in J_u(\true) }A_{ij} -  \sum_{v\notin J_u(\true)}A_{ij} \geq \ve \log n$ when $p > q$ and $\sum_{v\in J_u(\true) }A_{ij} -  \sum_{v\notin J_u(\true)}A_{ij} \leq -\ve \log n$ when $p < q$ where $\ve = o(1)$ (see~\eqref{eq:spec-disc-events}).
Thus it intuitively makes sense that the error rates of $\spec$ and $\best$ should be close.

We proceed with an error analysis of the genie-based estimator in Section \ref{sec:genie-error-analysis}, which will be used to complete the proof of Theorem~\ref{thm:spec-best-compared} in Section \ref{sec:genie-result-theorem}

\subsection{Error analysis of the genie-based estimator.}\label{sec:genie-error-analysis}
Next we analyze the error rate of the genie-based estimator~$\best$. Recall the error rate  $\err(\cdot)$ from~\eqref{eq:distance-defn} and $t_c(p_1,p_2,q)$ from \eqref{defn:CH-distance}. 
\begin{lemma} \label{lem:genie}
$\err(\best) =  n^{-(1+o(1))t/t_c(p_1,p_2,q)}$. 
\end{lemma}

\begin{proof}
Let $X:= \frac{1}{n} \sum_{i\in [n]} \ind{\best(i) \neq \true (i)} $. 
Then $\err(\best) = \E[\min \{X,1-X\}] $.
Note that, since $0\leq X\leq 1$ almost surely, we have 
\begin{equation*} 
   \E[X] -\E[\min\{X,1-X\}] \leq  \PR\Big(X>\frac{1}{2}\Big).
\end{equation*}
Therefore,
\begin{equation*}
\mathbb{E}[X] - \mathbb{P}\Big(X > \frac{1}{2}\Big) \leq \err(\best) = \E[\min \{X,1-X\}] \leq \mathbb{E}[X].
\end{equation*}
It suffices to show that $\mathbb{E}[X] = n^{-(1+o(1))t/t_c(p_1,p_2,q)}$ and $\mathbb{P}(X > \nicefrac{1}{2}) = o\left(\mathbb{E}[X]\right)$.\\

\noindent {\bf Computing $\mathbb{E}[X]$.}
Fix $u\in [n]$
 and let us compute $\PR(\best (u) \neq \true (u))$. Estimating $\true (u)$ is a binary Bayesian hypothesis testing problem, where the prior is given by $\PR(\true (u) = +1) = \PR(\true (u) = -1) = \frac{1}{2}$. 
We are given the observed edge-labeled graph $G$ and $\{\true (v): v\in [n] \setminus \{u\}\}$, which satisfies \eqref{eq:count-concentration} with $\ve = n^{-\frac{1}{3}}$ with probability at least $1 - 2\exp(-n^{\nicefrac{1}{3}}/6 ) =  1-n^{-\omega(1)}$. The genie-based estimator performs a Maximum A Posteriori (MAP) decoding rule based on the observed degree profiles $D = D(\true, u)$. Let $\cD:=\{d: d_i \leq \log^{\nicefrac{5}{4}} n, \forall i \in [4] \}$. By \eqref{eq:neighbor-concentration},
\begin{align*}
&\PR(\best (u) \neq \true (u)) = \frac{1}{2}\sum_{d\in \cD} \min \Big\{ \PR (D = d \mid \true(u) = +1), \PR (D = d \mid \true(u) = -1) \Big\} + n^{-\omega(1)}.
\end{align*}
Using Lemma~\ref{fact:stirling}, 
\begin{align*}
& \PR(\best (u) \neq \true (u)) \\
&\asymp \frac{1}{2}  \sum_{d\in \cD} \min \{M(n_1, d_1,d_2,p_1)\times M(n_2, d_3,d_4,q), M(n_1, d_1,d_2,q)\times M(n_2, d_3,d_4,p_2) \} + n^{-\omega(1)}\\ 
& \asymp \sum_{d\in \cD} \min \bigg\{P\Big(\frac{t p_1 \log n}{2}; d_1\Big) P\Big(\frac{t(1-p_1)\log n}{2}; d_2\Big)P\Big(\frac{t q \log n}{2}; d_3\Big)P\Big(\frac{t(1-q) \log n}{2}; d_4\Big), \\
& \hspace{2cm} P\Big(\frac{tq \log n}{2}; d_1\Big) P\Big(\frac{t(1-q) \log n}{2}; d_2\Big)P\Big(\frac{t p_2 \log n}{2}; d_3\Big)P\Big(\frac{t (1-p_2) \log n}{2}; d_4\Big)\bigg\} +n^{-\omega(1)}.
\end{align*}
We can now use \cite[Theorem 3]{AS15} with 
$c_1' = \frac{t}{2}(p_1, 1-p_1, q, 1-q)$, and  $c_2' = \frac{t}{2}(q,  1-q, p_2, 1-p_2)$ to conclude that 
\begin{eq}
\PR(\best (u) \neq \true (u)) & =  n^{-\Delta_t(c_1',c_2') (1+o(1))} + n^{-\omega(1)},
\end{eq}
where 
\begin{eq}\label{CH-distance}
\Delta_t(c_1',c_2') : = \max_{x\in [0,1]} \sum_{i} \big(x(c'_1)_i +(1-x)(c'_2)_i -(c'_1)_i^x (c'_2)_i^{1-x} \big) = \frac{t}{t_c(p_1,p_2,q)}.
\end{eq}
Thus, we have shown that, for any $u\in [n]$, 
\begin{eq}\label{error-best-u}
\mathbb{E}[X] = \PR(\best (u) \neq \true (u)) = n^{-(1+o(1))t/t_c(p_1,p_2,q)}. 
\end{eq}

\noindent {\bf Computing $\mathbb{P}(X > \nicefrac{1}{2})$.}
Note that 
\begin{align}\label{X-moment-2}
    \E[X^2] &= \frac{1}{n^2} \sum_{i\in [n]}\PR(\best (i) \neq \true (i))+ \frac{1}{n^2} \sum_{i,j\in [n]: i\neq j}\PR(\best (i) \neq \true (i),\ \best (j) \neq \true (j)).
\end{align}   
Fix any $i\neq j$ and let $\cF$ denote the minimum sigma-algebra with respect to which $D(i,\true)$ and $N(i)$ are measurable. Then the event $\{\best(i) \neq \true (i)\}$ is measurable with respect to $\cF$. 
Let $\cB$ be the event that $\{i,j\}$ is revealed or there exists a $v$ such that both $\{i,v\}$ and $\{v,j\}$ are revealed. 
Let us condition on $\cF$. Let $\best'(j)$ be the genie estimator computed on $G \setminus (N(i) \cup \{i\})$. On the event $\cB^c$, we have that $\best (j) = \best'(j)$, as  $D(j,\true)$ remains identical on $G$ and $G \setminus (N(i) \cup \{i\} )$.
Note that for all sufficiently large $n$, almost surely, 
\begin{align*}
    \PR(\cB \mid \cF) \ind{|N(i)| \leq \log^{\nicefrac{5}{4}} n} \leq \alpha (1+ \log^{\nicefrac{5}{4}} n)  \leq  \frac{\log^3 n}{n}, 
\end{align*}and therefore
\begin{align*}
    \PR(\best (j) \neq \true (j) \mid \cF)  \ind{|N(i)| \leq \log^{\nicefrac{5}{4}} n} &\leq \big(\PR(\best (j) \neq \true (j), \cB^c \mid \cF) + \PR(\cB \mid \cF) \big) \ind{|N(i)| \leq \log^{\nicefrac{5}{4}} n} \\
    &\leq \PR(\best' (j) \neq \true (j) \mid \cF) \ind{|N(i)| \leq \log^{\nicefrac{5}{4}} n}+ \frac{\log^3 n}{n}\\
    & = n^{-(1+o(1))t/t_c(p_1,p_2,q)}+ \frac{\log^3 n}{n},
\end{align*}
where the expression in the last step can be computed using identical arguments as \eqref{error-best-u}. 
Also, $\PR(|N(i)|>\log^{\nicefrac{5}{4}} n) \leq \e^{-\log^{\nicefrac{5}{4}} n}$ for all sufficiently large $n$.
Therefore, 
\begin{align*}
    &\PR(\best (i) \neq \true (i),\ \best (j) \neq \true (j)) \\
    & \leq  \E[\PR(\best (j) \neq \true (j) \mid \cF) \ind{|N(i)| \leq \log^{\nicefrac{5}{4}} n} \ind{\best (i) \neq \true (i)}]  + \PR(|N(i)| > \log^{\nicefrac{5}{4}} n)\\
    & \leq \Big(n^{-(1+o(1))t/t_c(p_1,p_2,q)} + \frac{\log^3 n}{n}\Big)\PR(\best (i) \neq \true (i)) + \e^{-\log^{\nicefrac{5}{4}} n} \\
    & \leq   n^{-\delta-(1+o(1))t/t_c(p_1,p_2,q)},
\end{align*} 
for some fixed $\delta>0$. 
Using \eqref{X-moment-2}, we have that $\E[X^2] =  n^{-\delta-(1+o(1))t/t_c(p_1,p_2,q)}$, and by Markov's inequality, \[\PR(X\geq 1/2) = \PR(X^2 \geq 1/4) \leq 4n^{-\delta-(1+o(1))t/t_c(p_1,p_2,q)} = o(\E[X]). \]
\end{proof}

\subsection{Comparing the spectral and genie estimators}\label{sec:genie-result-theorem}
\begin{proof}\emph{(Proof of Theorem~\ref{thm:spec-best-compared}).}
We prove the claim for the case $p > q$; the case $p < q$ follows by replacing $A$ by $-A$. By Lemma~\ref{lem:genie} and \eqref{eq:critical-vals-compare}, we have $\err(\best) =  n^{-(1+o(1))t/t_c(p,q)}$.
To analyze the error rate of the spectral algorithm, fix $\varepsilon>0$ and let
\[\cB_n:=\left\{\min_{s\in\{\pm 1\}}\left \Vert u_1-s\frac{ Au_1^\star}{\lambda_1^\star} \right \Vert_{\infty}\leq \frac{\varepsilon}{\sqrt{n}} \right\} \text{~~and~~} \mathcal{C}_n:=\left\{\left| n_1(\true) - \frac{n}{2}\right| \leq n^{\frac{2}{3}}, \left| n_2(\true) - \frac{n}{2}\right| \leq n^{\frac{2}{3}} \right\}.\] 
Let $s\in \{\pm 1\}$ be the sign for which $\|u_1-s Au_1^\star/\lambda_1^\star \|_{\infty}$ is minimized. 
By Lemma~\ref{lemma:corollary-3.1-abbe}, $\PR(\cB_n) = 1-O(n^{-3})$. 
If the spectral algorithm is not able to classify $i$ correctly, then $s\true (i)(u_1)_i\leq 0$. This implies that, on $\cB_n$, we have $\sqrt{n}\true(i)Au_1^\star/\lambda_1^\star \leq \varepsilon$. 
Using \eqref{eq:count-concentration}~and~\eqref{eq:spec-disc-events}, for any $i\in [n]$, 
\begin{eq} \label{misclassify-spec}
    \PR( \spec (i) \neq s\true (i)) &\leq \PR((\cB_n \cap \mathcal{C}_n)^c) + \mathbb{P}\bigg(\sum_{j \in J_i(\true)} A_{ij} - \sum_{j \notin J_i(\true)} A_{ij} \leq \ve \log(n) ~\big|~ \mathcal{C}_n \bigg) \\ 
    & \leq O(n^{-3}) + n^{-(1+o(1)) t/t_c(p,q)},
\end{eq}
where the final step uses Lemma~\ref{lemma:supporting-3}. The proof follows by taking an average over $i$.
\end{proof}

\section{Analysis of the two-step estimator} \label{sec:two-step} 
In order to establish Theorems \ref{theorem:two-stage-spectral} and \ref{theorem:two-stage-degree}, we introduce the notion of a \emph{good estimator}. Consider the exact recovery problem on $\SBM(p_1, p_2,q,t)$, possibly with $p_1\neq p_2$. We will show that, given a ``good'' initial estimator $\hsig$, the clean-up procedure defined by $\hat{X}(\hsig)$ produces an estimator that recovers $\sigma_0$ exactly when $t > t_c(p_1,p_2,q)$, and otherwise compares favorably with the genie estimator. Recall the definition of $N(u)=\{v: A_{uv} \neq 0\}$, and define $M(\hsig,\true)$ to be the set of misclassified vertices under the optimal choice of the global flip, i.e. $M(\hsig, \true)$ is the smaller of
\[\{v : \hsig(v) \neq \true(v) \} ~~\text{and}~~ \{v : -\hsig(v) \neq \true(v) \}.\]
\begin{definition} \label{defn:good-estimator} \normalfont
We say that  $\hsig$ is a \emph{good estimator} if:
\begin{enumerate}[(i)]
    \item \label{good-estimator-2} There exists $L\geq 1$ (fixed) such that, for all sufficiently large $n$,
    \begin{eq}
    \PR(|N(u) \cap M(\hsig,\true)| \leq L  ) \geq 1-o(1/n) \quad \forall u\in [n],
    \end{eq}i.e., any vertex can have at most $O(1)$ many misclassified neighbors.
    \item \label{good-estimator-3} There exists a constant $\ve>0$ such that, for all $u\in[n]$, 
    \[\max_{d:||d||_1\le\ln^{\nicefrac{5}{4}}(n)}\PR\left(N(u) \cap M(\hsig,\true) = \varnothing|D(\sigma_0,u)=d\right)=1-O(n^{-\ve}),
    \]i.e., vertices have a negligible probability of having a misclassified revealed neighbor/non-neighbor given the correct degree profile. 
\end{enumerate}
\end{definition}
Note that Definition \ref{defn:good-estimator}~\ref{good-estimator-3} implies that there exists $(\varepsilon_n)_{n\geq 1} \subset (0,\infty)$ with $\lim_{n\to\infty} \varepsilon_n = 0$, such that 
\begin{eq}\label{good-estimator-implied}
\PR \big(\# \{i: N(i) \cap M(\hsig,\true) = \varnothing\} \geq (1-\varepsilon_n) n \big) \to 1,
\end{eq}
i.e, $\hsig$ classifies all but $o(n)$ many vertices correctly.

Comparing the two-step estimator $\hat{X}(\hsig)$ from \eqref{best:est-sym} to the genie-based estimator \eqref{eq:genie-expression-general}, we see that a good estimator acts as a proxy for the true labels; just as the true labels are used to compute the genie-based estimator, so the good estimator labels are used by the two-step estimator. The next result provides recovery guarantees for good estimators.
\begin{theorem}\label{thm:two-step}
Suppose that $\hsig$ is a good estimator. If $t>t_c(p_1,p_2,q)$,  then $\hat{X}(\hsig)$ achieves exact recovery. Moreover, for any $t>0$, 
\begin{eq}\label{eq:two-step-error}
\err(\hat{X}(\hsig) ) =  (1+o(1))\err(\best) + o(1/n). 
\end{eq}
\end{theorem}
Given Theorem \ref{thm:two-step}, it suffices to show that $\spec$ and $\degree$ are good estimators in order to establish Theorems \ref{theorem:two-stage-spectral} and \ref{theorem:two-stage-degree}. In Section \ref{sec:good-estimator-guarantee}, we prove Theorem \ref{thm:two-step}, and in Section \ref{sec:good-estimators} we prove Theorems \ref{theorem:two-stage-spectral} and \ref{theorem:two-stage-degree}.

\begin{remark} \normalfont
In the case where $p_1=p_2$, the clean up step is essentially just a way of removing noise. $\spec$ is essentially just the primary eigenvector of $A$ with some noise removed, which means that $\hX(\spec)$ is a rounded version of $A\spec$, which can be viewed as the eigenvector with more noise removed. The result is an estimator which is similar to the eigenvector but with less susceptibility for its guess of a vertex's community to be influenced by atypicalities in its neighbors' degree profiles.

In the asymmetric case $\degree$ is a significantly worse estimator because it does not take into account any information on the communities of the vertices a target vertex's edges are to. Nevertheless, it classifies vertices with a sufficiently high accuracy for the two step approach to work. Namely, the initial estimate provides good estimates of the vertices' degree profiles, such that the clean-up step correctly selects the most likely community for the vertices, given their initially estimated degree profiles.\footnote{We could have used a spectral algorithm for the first step, but it would be harder to compute without giving any real advantage.}
\end{remark}

\subsection{Error rate guarantees for good estimators.}\label{sec:good-estimator-guarantee}
Suppose we are given a good estimator $\hsig$. 
Define $V_{\sss \mathrm{good}}:= \{i: N(i) \cap M(\hsig,\true) = \varnothing \}$. 
That is, $V_{\sss \mathrm{good}}$ is the set of vertices whose neighbors are all correctly classified by $\hsig$. We denote the rest as $V_{\sss \mathrm{bad}} = [n] \setminus V_{\sss \mathrm{good}}$. 
Note that if $i\in V_{\sss \mathrm{good}}$, then the two step estimator produces the same assignment for $i$ as the genie estimator. 
The next lemma handles the case that $i\in V_{\sss \mathrm{bad}}$ and the estimator $\hat{X}(\hat{\sigma},i)$ is incorrect.
 
\begin{lemma}\label{lem:spec-bad}
Let $\hsig$ be a good estimator. For any $i\in [n]$, we have
\begin{align*}
\PR(\hX (\hsig, i) \neq \true (i), i\in V_{\sss \mathrm{bad}}) =o\left(\PR(\hsig_{\sss \mathrm{Best}}(i)\neq \true (i))+1/n\right).
\end{align*}
\end{lemma}
\begin{proof}
Fix a vertex $i$, and recall the degree profile notation $D(\sigma,i) = (D_k( \sigma,i))_{k=1}^4$. Define $\cD:= \{d \in \mathbb{Z}_+^4: 
\|d\|_1 \leq \log^{\nicefrac{5}{4}}n\}$. 
By \eqref{eq:neighbor-concentration}, 
$\PR(D(\true,i) \notin \cD) = o(1/n)$.
First, we claim that for any $d,d'\in \cD$ satisfying $\|d-d'\|_1 \leq L$, 
\begin{eq}\label{prob-ratio}
\frac{\mathbb{P}\left(D(\true,i) = d'\right)}{\mathbb{P}\left(D(\true,i) = d\right)} \leq c_0 \log^{2L} n, 
\end{eq}
where $c_0>0$ is a constant that depends only on $p_1$, $p_2$, $q$, $t$, and $L$. Indeed, by Lemma~\ref{fact:stirling},  
$$(1+o(1))\min \{A(d),B(d)\}\leq\PR(D(\true,i) = d) \leq (1+o(1)) \max\{A(d),B(d)\}$$
where 
\begin{align*}
    A(d)&:=  P\Big(\frac{t p_1 \log n }{2}; d_1\Big)P\Big(\frac{t (1-p_1) \log n }{2}; d_2\Big)P\Big(\frac{t q \log n }{2}; d_3\Big)P\Big(\frac{t (1-q) \log n }{2}; d_4\Big)\\
    B(d)&:=  P\Big(\frac{t q \log n }{2}; d_1\Big)P\Big(\frac{t (1-q) \log n }{2}; d_2\Big)P\Big(\frac{t p_2 \log n }{2}; d_3\Big)P\Big(\frac{t (1-p_2) \log n }{2}; d_4\Big). 
\end{align*}
Using $\|d-d'\|_1 \leq L$, \eqref{prob-ratio} now follows.
Next, we set up some notation to complete the proof. Define $\cS_0(i) \subset \cS$ to be the set of $\sigma$ such that $|N(i) \cap M(\sigma,\true)|\leq L$. By Definition~\ref{defn:good-estimator}~\ref{good-estimator-2}, we have that $\PR(\hsig \notin \cS_0(i)) = o(1/n)$.  
For every $d$, let \[\overline{\sigma}(d,i) := \argmax_{r\in \{\pm 1\}}\PR(\true(i) = r \mid D(\true,i) = d)\] 
be the most likely community assignment of $i$ given the observed degree profile $d$ of~$i$. We set $\overline{\sigma}(d,i) = +1$ if they are equally likely. 
Define $\cB (d,L):=\{d'\in \cD: \|d-d'\|_1 \leq L \text{ and } \overline{\sigma}(d',i) \neq \true(i)\}$. 
In other words, $\cB (d,L)$ is the set of degree profiles near $d$ on which even the best estimator makes a mistake.
Note that if we have $\hX (\hsig, i) \neq \true (i)$, $i\in V_{\sss \mathrm{bad}}$,  and $\hsig \in \cS_0(i)$, then it must be the case that $D(\true,i)$ has a degree profile $d$ such that $\cB (d,L) \neq \varnothing$. Therefore, using Definition~\ref{defn:good-estimator}~\ref{good-estimator-3} and \eqref{prob-ratio}, 
\begin{align*}
&\PR(\hX (\hsig, i) \neq \true (i), i\in V_{\sss \mathrm{bad}})\\
& \leq  \PR(\hX (\hsig, i) \neq \true (i), i\in V_{\sss \mathrm{bad}}, D(\true,i) \in \cD, \hsig\in \cS_0(i)) +o(1/n) \\
&\le \sum_{d \in \cD:\cB (d,L) \neq \varnothing } \PR\left(i\in V_{\sss \mathrm{bad}},D(\sigma_0,i)=d\right)+o(1/n)\\
&\le\sum_{d \in \cD:\cB (d,L) \neq \varnothing } cn^{-\ve}\PR\left(D(\sigma_0,i)=d\right) +o(1/n)\\
&= cn^{-\ve}\PR(D(\sigma_0,i) \in \cD \text{ and }\cB (D(\sigma_0,i),L)\neq \varnothing ) +o(1/n)\\
&\le cn^{-\ve}\sum_{d': \overline{\sigma}(d',i)\ne \true(i)} \PR(\|D(\sigma_0,i)- d'\|_1\le L) +o(1/n) 
\\
&\le cn^{-\ve}\sum_{d': \overline{\sigma}(d',i)\ne \true(i)} c_0\ln^{2L}(n)\PR(D(\sigma_0,i)=d') +o(1/n)\\
&=  cc_0n^{-\ve} \ln^{2L}(n)\times \PR(\hsig_{\sss \mathrm{Best}}(i)\neq \true (i)) +o(1/n)\\
&=o\left(\PR(\hsig_{\sss \mathrm{Best}}(i)\neq \true (i))+1/n\right) 
\end{align*}
\end{proof}

\begin{proof}\emph{(Proof of Theorem~\ref{thm:two-step}).} Fix $i\in [n]$. 
Recall that $V_{\sss \mathrm{good}}:= \{i: N(i) \cap M(\hsig,\true) = \varnothing \}$, and  $V_{\sss \mathrm{bad}} = [n] \setminus V_{\sss \mathrm{good}}$. 
Using Lemma~\ref{lem:spec-bad}, we have
\begin{align*}
\PR(\hX (\hsig, i) \neq \true (i) ) = (1+o(1)) \PR(\best(i) \neq \true(i) ) + o(1/n). 
\end{align*}
Taking an average over $i$ and using $\err(\hX (\hsig) ) \leq \frac{1}{n}\sum_{i\in [n]}\PR(\hX (\hsig, i) \neq \true (i) ) $ and $\err(\best) =  \frac{1}{n}\sum_{i\in [n]}\PR(\best(i) \neq \true (i) ) + o(1/n) $ (see the argument after \eqref{error-best-u}),  the proof follows. 
\end{proof}

\subsection{Good estimators}\label{sec:good-estimators}
\begin{proof}\emph{(Proof of Theorem~\ref{theorem:two-stage-spectral}).}
Due to Theorem \ref{thm:two-step}, it suffices to show that $\spec$ is a good estimator.
We will first verify Definition~\ref{defn:good-estimator}~\ref{good-estimator-2}. Let $L\geq 1$ be fixed (to be chosen later). 
For any $V\subset [n]$, let 
$$D_V(j):= \sum_{k\in V^c: k\sim j}A_{jk} - \sum_{k\in V^c: k\not\sim j} A_{jk}. $$
By the previous analysis (cf.~\eqref{misclassify-spec}), 
\begin{eq}
\PR(\spec(j)  \text{ misclassifies}) \leq \PR(D_{\varnothing}(j) \leq \varepsilon \log n) + o(1/n),
\end{eq}
 where $\varepsilon  = o(1)$ but $\varepsilon \log n \to \infty$.
Additionally, if $|V| \leq L$ for some fixed $L\geq 0$, then $D_V(j) \leq \varepsilon \log n+ O(L) \leq  2 \varepsilon \log n$ for large enough $n$. Thus,
\begin{align*}
&\PR(|N(i) \cap M(\spec,\true)| > L)\\
&\leq \PR(|N(i)| > \log ^{\nicefrac{5}{4}}n) + \E\big[\PR\big(|N(i) \cap  M(\spec,\true) | > L \mid N(i), |N(i)| \leq  \log ^{\nicefrac{5}{4}}n\big)\big]\\
&\leq \e^{-\log^{\nicefrac{5}{4}} n} +  \E\bigg[\sum_{j_0,\dots,j_L\in N(i)}\PR\big(j_0,\dots, j_{L} \in N(i) \cap  M(\spec,\true) \mid N(i), |N(i)| \leq  \log^{\nicefrac{5}{4}} n\big)\bigg]\\
&\leq \e^{-\log^{\nicefrac{5}{4}} n} + \binom{\log^{\nicefrac{5}{4}} n}{L+1}  \E\Big[\max_{j_0,\dots,j_L}\PR\big(j_0,\dots, j_{L} \in N(i) \cap  M(\spec,\true) \mid N(i), |N(i)| \leq  \log^{\nicefrac{5}{4}} n\big)\Big],
\end{align*}
where we have applied \eqref{eq:neighbor-concentration}. Taking $V = \{j_0,\dots,j_L,i\}$ yields
\begin{align*}
&\PR\big(j_0,\dots, j_{L} \in N(i) \cap M(\spec,\true) \mid N(i), |N(i)| \leq  \log ^{\nicefrac{5}{4}} n\big)\\
&\leq \PR\big(D_V(j_l) \leq 2 \varepsilon \log n, \ \forall 0\leq l\leq L \mid N(i), |N(i)| \leq  \log ^{\nicefrac{5}{4}} n\big). 
\end{align*}
Now, note that $D_V(j_l)$ for different $l$ depend on disjoint sets of random variables. Thus, Lemma~\ref{lemma:supporting-3} gives
\begin{align*}
\PR\big(D_V(j_l) \leq 2 \varepsilon \log n, \ \forall 0\leq l\leq L \mid N(i), |N(i)| \leq  \log ^{\nicefrac{5}{4}} n\big) \leq n^{-c(L+1)},
\end{align*}
for some constant $c>0$. 
Therefore, 
\begin{align*}
\PR(|N(i) \cap M(\spec,\true)| > L)\leq \e^{-\log^{\nicefrac{5}{4}}n} + (\log n)^{\nicefrac{5(L+1)}{4}} n^{-c(L+1)} = o(1/n),
\end{align*}
which verifies Definition~\ref{defn:good-estimator}~\ref{good-estimator-2} by taking $L$ to be a large fixed constant. To verify Definition~\ref{defn:good-estimator}~\ref{good-estimator-3}, we can repeat the above argument for conditional probabilities with $L=0$. 
Indeed, for any $d$ with $\|d\|_1\leq \log^{\nicefrac{5}{4}}n$, 
\begin{align*}
\PR(|N(i) \cap M| \geq 1 \mid D(\true,i) = d) \leq(\log n)^{\nicefrac{5}{4}}\cdot  n^{-c}, 
\end{align*}
and the proof follows. 
\end{proof}

\begin{proof}\emph{(Proof of Theorem~\ref{theorem:two-stage-degree}).} 
Due to Theorem \ref{thm:two-step}, it suffices to show that $\degree$ is a good estimator. We have
\begin{align}\label{degree-expts}
\mathbb{E}[\deg(j) \mid \true]
 &= \begin{cases}
\frac{t \log(n)}{2}\left(p_1 + q \right) & \quad \text{if }\sigma_0(j) = +1,\\
\frac{t \log(n)}{2}\left(p_2 + q \right) &  \quad \text{if } \sigma_0(j) = -1.
\end{cases}
\end{align}
Suppose $p_1 > p_2$ without loss of generality. We bound 
\begin{align*}
&\mathbb{P}\Big(\deg(j) < \frac{t \log(n)}{4}\left(p_1 + p_2 + 2q \right)  \ \Big\vert\ \true(j) = +1\Big) \\
&= \mathbb{P}\Big(\deg(j) < \Big(1 - \frac{p_1 - p_2}{2(p_1 + q)} \Big) \mathbb{E}[\deg(j) \mid \sigma_0(j) = +1] \ \Big\vert\ \true(j) = +1 \Big)\\
&\leq \exp\bigg(-\frac{\mathbb{E}[\deg(j) \mid \sigma_0(j) = +1]}{2} \Big(\frac{p_1 - p_2}{2(p_1 + q)}\Big)^2\bigg)\\
&= \exp\bigg(-\frac{t \log(n)(p_1 + q)}{4} \Big(\frac{p_1 - p_2}{2(p_1 + q)}\Big)^2\bigg)\\
&=\exp\bigg(-\frac{t \log(n)(p_1 - p_2)^2}{16 (p_1 + q)} \bigg).
\end{align*}
Similarly,
\begin{align*}
&\mathbb{P}\Big(\deg(j) \geq \frac{t \log(n)}{4}(p_1 + p_2 + 2q ) \ \Big\vert\ \true(j) = -1 \Big) \\
&= \mathbb{P}\Big(\deg(j) \geq \left(1 +  \frac{p_1 - p_2}{2(p_2 + q)}\right) \mathbb{E}[\deg(j) \mid \sigma_0(j)=-1]  \ \Big\vert\ \true(j) = -1\Big)\\
&\leq \exp \bigg(- \frac{\left(\frac{p_1 - p_2}{2(p_2 + q)}\right)^2}{2 + \frac{p_1 - p_2}{2(p_2 + q)}} \mathbb{E}[\deg(j) \mid \sigma_0(j) = -1] \bigg)\\
&= \exp \bigg(- \frac{t \log n (p_2 + q)}{2} \frac{ \big(\frac{p_1 - p_2}{2(p_2 + q)}\big)^2}{2 + \frac{p_1 - p_2}{2(p_2 + q)}}\bigg)\\
&=\exp \left(- \frac{t \log n (p_1 - p_2)^2 }{4(p_1 + 3p_2 + 4q)}\right).
\end{align*}
Therefore, we have shown that any given vertex $j$ is misclassified with probability at most $n^{-\beta}$, where $\beta > 0$ is a constant. Fix $i \in [n]$. We wish to take a union bound over $j\in N(i)$ in order to bound the probability that some revealed neighbor of $i$ is misclassified. However, if a vertex $j$ is a neighbor of $i$, then its degree is inflated by $1$. Since the discrepancy is negligible at the $\log(n)$ scale, we may take $\beta$ slightly smaller. Then, for any $d$ such that $\|d\|_1\leq \log^{\nicefrac{5}{4}}n$, 
\[\mathbb{P}\left(N(i) \cap M(\degree, \true) \neq \varnothing ~\big|~ D(\true,i) = d \right) \leq 
\log^{\nicefrac{5}{4}}(n) n^{-\beta} \leq n^{-\frac{\beta}{2}}\]
for $n$ large enough, satisfying Definition~\ref{defn:good-estimator}~\ref{good-estimator-3}. 

Next, take $L > \frac{2}{\beta}$. By a Union Bound,
\begin{align*}
\mathbb{P}\left(|N(i) \cap M(\degree, \true)| \geq L \right) &\leq \binom{\log^{\nicefrac{5}{4}}(n)}{L} n^{-\beta L} + n^{-\omega(1)}\\
&\leq \log^{\nicefrac{5L}{4}}(n) n^{-\beta L}  + n^{-\omega(1)} \leq  n^{-\frac{3}{2}},
\end{align*}
for $n$ large enough, which verifies Definition~\ref{defn:good-estimator}~\ref{good-estimator-2}.
\end{proof}

\section{Failure of the spectral algorithm}\label{sec:failure}
Similarly to the proof of Theorem \ref{theorem:exact-recovery-spectral}, we apply the entrywise approach to analyze the eigenvectors of $A$. Any linear combination of $u_1$ and $u_2$ can then be approximated by $A$ times a vector $z$ that is constant on each community. Then we can characterize all spectral algorithms as thresholding weighted degree profiles by a \emph{score} that depends on the vector $z$. Sufficiently close to the threshold, we exhibit certain degree profiles that are likely enough to occur, for which there is no way to pick the encoding value $y$ or the weights $z$ in order to separate the communities. Specifically, if we take the degree profile that is half way between the expected degree profiles for vertices in the two communities, there will always be a way to perturb it such that vertices with that degree profile are likely to exist in one community and the spectral algorithm classifies them as being in the other. Unlike in the symmetric case, the class of spectral algorithms considered here does not have enough degrees of freedom to allow correct separation of vertices with such degree profiles.

\begin{proof}\emph{(Proof of Theorem~\ref{theorem:antisymmetric-spectral-failure}).}
Throughout the proof, we will condition on $\sigma_0$ satisfying $\left| n_1 (\true)- \frac{n}{2}\right| \leq n^{\frac{2}{3}}$ and $\left| n_2 (\true) - \frac{n}{2}\right| \leq n^{\frac{2}{3}}$.

Let $u_1$ and $u_2$ be the top two eigenvectors of $A$. By Proposition~\ref{lemma:corollary-3.1-abbe} we have that with probability $1-O(n^{-3})$,
\[\left\Vert u_1 - s_1 \frac{A u_1^\star}{\lambda_1^\star} \right\Vert_{\infty} \leq \frac{C}{\sqrt{n} \log \log n}~~~\text{and}~~~\left \Vert u_2 - s_2 \frac{A u_2^\star}{\lambda_2^\star} \right \Vert_{\infty} \leq \frac{C}{\sqrt{n} \log \log n},\]
for some $s_1, s_2 \in \{-1,1\}$ and some constant $C>0$. Then for any $c_1, c_2$ we also have
\[\left \Vert c_1u_1+c_2u_2 - A\left(\frac{s_1 c_1}{\lambda_1^\star} u_1^\star+\frac{s_2 c_2}{\lambda_2^\star} u_2^\star\right) \right \Vert_{\infty} \leq \frac{C(|c_1| + |c_2|)}{\sqrt{n} \log \log n}.\]
By Lemma~\ref{lem:evals-Astar}, the vector $(\frac{s_1c_1}{\lambda_1^\star} u_1^\star+\frac{s_2c_2}{\lambda_2^\star} u_2^\star)$ will assign all vertices in Community 1 some value $z_1$ and all vertices in Community 2 some other value $z_2$. We can assume without loss of generality that $z_1^2+z_2^2=1$. We conclude that a spectral algorithm based on a linear combination of the top two eigenvectors will classify a vertex with degree profile $(d_1,d_2,d_3,d_4)$ as being in one community if \[z_1 d_1-yz_1d_2+z_2d_3-yz_2 d_4>r+\Omega\left(\log(n)/\log(\log(n))\right)\] and the opposite community if \[z_1 d_1-yz_1d_2+z_2d_3-yz_2 d_4<r-\Omega\left(\log(n)/\log(\log(n))\right)\] for some threshold $r$. To show that the spectral algorithm fails for $t$ sufficiently close to the threshold, it suffices to show that there is no choice of $z_1$, $z_2$ and $y$ for which thresholding the score $z_1 d_1-yz_1d_2+z_2d_3-yz_2 d_4$ separates the communities successfully.

We will first show that
\begin{equation}
r=(z_1-yz_2)\sqrt{p/8}t\ln(n)+(z_2-yz_1)\sqrt{(1-p)/8}t\ln(n)+o(\log(n)) \label{eq:claimed-threshold} 
\end{equation}
and both $y$ and $z_1z_2$ are positive. We can then assume without loss of generality that $z_1,z_2>0$ and the algorithm classifies a vertex as being in Community 1 if its score is larger than $r+\Omega(\log(n)/\log(\log(n)))$, and otherwise classifies it as being in Community 2.
We will then show that there exist constants $\delta>\delta'>0$ such that there are vertices with degree profile
\begin{equation}
\left(\sqrt{p/8}t\ln(n)-\delta\ln(n),\sqrt{(1-p)/8}t\ln(n)-\delta'\ln(n),\sqrt{(1-p)/8}t\ln(n),\sqrt{p/8}t\ln(n)\right) \label{eq:profile-1}
\end{equation}
in Community 1 and vertices with degree profile 
\begin{equation}
\left(\sqrt{p/8}t\ln(n),\sqrt{(1-p)/8}t\ln(n),\sqrt{(1-p)/8}t\ln(n)-\delta'\ln(n),\sqrt{p/8}t\ln(n)-\delta\ln(n)\right) \label{eq:profile-2}
\end{equation}
in Community 2. In order to classify the former vertices correctly, the spectral algorithm would need to have $y>\delta/\delta'-o(1)$ and to classify the later vertices correctly the algorithm would need to have $y<\delta'/\delta+o(1)$, but for large $n$ those cannot both be true.

It remains to establish the claims regarding the threshold value, and the existence of common degree profiles. To this end, we compute the probability of a given degree profile for $\true(i) = +1$ and $\true(i) = -1$.
Given constants $c_1$, $c_2$, $c_3$, and $c_4$, set $c'_1=c_1-t\sqrt{p/8}$, $c'_2=c_2-t\sqrt{(1-p)/8}$, $c'_3=c_3-t\sqrt{(1-p)/8}$, $c'_4=c_4-t\sqrt{p/8}$, and $\epsilon=\min\{\sqrt{p/8}t,\sqrt{(1-p)/8}t,c_1,c_2,c_3,c_4\}$. Computing the probability of given profiles using Lemma \ref{fact:stirling}, 
\begin{eq}
&\mathbb{P}\left(D(\sigma_0,i) = \log(n)\left(c_1, c_2, c_3, c_4 \right) \right)  = \mathbb{P}\left(D(\sigma_0,j) = \log(n)\left(c_4, c_3, c_2, c_1 \right) \right)  \\
& \asymp P\left(\frac{tp \log n}{2}; c_1 \log n \right) P\left(\frac{t(1-p) \log n}{2}; c_2 \log n \right) P\left(\frac{t \log n}{4}; c_3 \log n \right)P\left(\frac{t \log n}{4}; c_4 \log n \right) \\
& =\frac{n^{-t} (\frac{t\ln(n)}{2})^{(c_1+c_2+c_3+c_4)\ln(n)}}{(c_1\ln(n))!(c_2\ln(n))!(c_3\ln(n))!(c_4\ln(n))!} \times  p^{c_1\ln(n)}(1-p)^{c_2\ln(n)}(1/2)^{(c_3+c_4)\ln(n)}\\
&\approx \frac{n^{-t} (\frac{\e t}{2})^{(c_1+c_2+c_3+c_4)\ln(n)}}{c_1^{c_1\ln(n)}c_2^{c_2\ln(n)}c_3^{c_3\ln(n)}c_4^{c_4\ln(n)}} \times p^{c_1\ln(n)}(1-p)^{c_2\ln(n)}(1/2)^{(c_3+c_4)\ln(n)} \\
&= \frac{n^{-t} \big[(\frac{\e^2t^2p}{8})^{(c_1+c_4)}(\frac{\e^2t^2(1-p)}{8})^{(c_2+c_3)}(2p)^{(c_1-c_4)}\left(2(1-p)\right)^{(c_2-c_3)}\big]^{\ln(n)/2} }{n^{c_1\ln(c_1)}n^{c_2\ln(c_2)}n^{c_3\ln(c_3)}n^{c_4\ln(c_4)}}. \label{eq:degree-profile-probability}
\end{eq}
\normalsize
To bound \eqref{eq:degree-profile-probability}, we will use the fact that for any $a,b > 0$,
\begin{equation}
a\ln(a)\le b\ln(b)+(a-b)\ln(\e  b)+\frac{(a-b)^2}{2\min(a,b)}, \label{eq:log-fact} 
\end{equation}
which follows from Taylor expansion and the fact that the second derivative of $x\ln(x)$ is at most $1/\min(a,b)$ between $a$ and $b$.
Taking $a = c_1$ and $b = c_1 - c_1' = t\sqrt{\nicefrac{p}{8}}$, in \eqref{eq:log-fact}, we have
\[c_1 \log(c_1) \leq t\sqrt{\frac{p}{8}} \log \left(t\sqrt{\frac{p}{8}} \right) + c_1'\log\left(\e t\sqrt{\frac{p}{8}}\right) + \frac{c_1'^2}{2\ve }  \]
and therefore
\begin{align*}
&\frac{1}{n^{c_1 \log(c_1)}} \bigg(\e t \sqrt{\frac{p}{8}} \bigg)^{c_1 \log(n)} \\
&\geq  \exp\left[-\log(n)\left(t\sqrt{\frac{p}{8}} \log \left(t\sqrt{\frac{p}{8}} \right) + c_1'\log\left(\e t\sqrt{\frac{p}{8}}\right)  - c_1 \log\left( \e t \sqrt{\frac{p}{8}}\right) \right) \right]n^{-\frac{c_1'^2}{2\ve}} \\
&=  \exp\left[-\log(n)\sqrt{\frac{p}{8}} \left(t \sqrt{\frac{p}{8}} + c_1' - c_1 \right) + t\log(n) \sqrt{\frac{p}{8}}  \right]n^{-\frac{c_1'^2}{2\ve}} \\
&=e^{t\sqrt{\frac{p}{8}} \log(n)} n^{-\frac{c_1'^2}{2\ve}}.
\end{align*}
Similarly,
\begin{align*}
\frac{1}{n^{c_2 \log(c_2)}} \bigg(\e t \sqrt{\frac{1-p}{8}} 
\bigg)^{c_2 \log(n)} 
&\geq \e^{t\sqrt{\frac{1-p}{8}}\log(n)}n^{-\frac{c_2^2}{2\ve}}\\
\frac{1}{n^{c_3 \log(c_3)}} \bigg(\e t \sqrt{\frac{1-p}{8}} \bigg)^{c_3 \log(n)} &\geq \e^{t\sqrt{\frac{1-p}{8}}\log(n)}n^{-\frac{c_3^2}{2\ve}}\\
\frac{1}{n^{c_4 \log(c_4)}} \bigg(\e t \sqrt{\frac{p}{8}} \bigg)^{c_2 \log(n)} &\geq \e^{t\sqrt{\frac{p}{8}}\log(n)}n^{-\frac{c_4^2}{2\ve}}.
\end{align*}
Next, for simplicity, let us denote $t_0 = t_c(p,1-p,\nicefrac{1}{2})$, and  we use the fact that the minimum is attained in \eqref{CH-distance} for $x = 1/2$ when $p_1 = 1-p_2$ and $q = \nicefrac{1}{2}$.
Substituting the above inequalities into \eqref{eq:degree-profile-probability}, and denoting $$Z_n = n^{-((c'_1)^2+(c'_2)^2+(c'_3)^2+(c'_4)^2)/2\ve} \times  (2p)^{(c'_1-c'_4)\ln(n)/2}\left(2(1-p)\right)^{(c'_2-c'_3)\ln(n)/2},$$ we obtain
\begin{eq}\label{eq:degree-profile-lower-bound}
&\mathbb{P}\left(D(\sigma_0,i) = \log(n)\left(c_1, c_2, c_3, c_4 \right) \right) \gtrapprox n^{-t} \e^{t\sqrt{p/2} \ln(n)}\e^{t\sqrt{(1-p)/2} \ln(n)} \times Z_n = n^{-t/t_0} \times Z_n,
\end{eq}
where $a_n \gtrapprox b_n$ means there is some constant $c$ such that $a_n =\Omega(b_n\log^c(n))$.
When $t = t_0$, observe that
$\mathbb{P}\left(D(\sigma_0,i) = \log(n) (c_1, c_2, c_3, c_4)\right)$ cannot be less than $Z_n/n$ by more than a polylogarithmic factor. Let 
\begin{align*}
c_1 = \sqrt{\frac{p}{8}} t_0 + \delta, \quad 
c_2 = c_3= \sqrt{\frac{(1-p)}{8}} t_0, \quad
c_4 =  \sqrt{\frac{p}{8}} t_0.
\end{align*}
If $\delta > 0$ is sufficiently small, with high probability there will be vertices with degree profiles of
\[\log(n)(\sqrt{p/8}t_0+\delta,\sqrt{(1-p)/8}t_0,\sqrt{(1-p)/8}t_0,\sqrt{p/8}t_0)\]
in Community 1, since \eqref{eq:degree-profile-lower-bound} is $\omega\left(\frac{1}{n}\right)$. In fact, the number of such vertices will be $\Omega\left(n^c\right)$ with high probability for some $c > 0$. A similar argument shows that for $\delta > 0$ small enough, there will additionally be vertices  with degree profiles of
\begin{align*}
&\log(n)(\sqrt{p/8}t_0,\sqrt{(1-p)/8}t_0-\delta,\sqrt{(1-p)/8}t_0,\sqrt{p/8}t_0)\\
&\log(n)(\sqrt{p/8}t_0,\sqrt{(1-p)/8}t_0,\sqrt{(1-p)/8}t_0+\delta,\sqrt{p/8}t_0)\\
&\log(n)(\sqrt{p/8}t_0,\sqrt{(1-p)/8}t_0,\sqrt{(1-p)/8}t_0,\sqrt{p/8}t_0-\delta)
\end{align*} in Community 1 and vertices with degree profiles of \begin{align*}
&\log(n)(\sqrt{p/8}t_0-\delta,\sqrt{(1-p)/8}t_0,\sqrt{(1-p)/8}t_0,\sqrt{p/8}t_0)\\
&\log(n)(\sqrt{p/8}t_0,\sqrt{(1-p)/8}t_0+\delta,\sqrt{(1-p)/8}t_0,\sqrt{p/8}t_0)\\
&\log(n)(\sqrt{p/8}t_0,\sqrt{(1-p)/8}t_0,\sqrt{(1-p)/8}t_0-\delta,\sqrt{p/8}t_0)\\
&\log(n)(\sqrt{p/8}t_0,\sqrt{(1-p)/8}t_0,\sqrt{(1-p)/8}t_0,\sqrt{p/8}t_0+\delta)
\end{align*}
in Community 2, when $t = t_0$. If $t = t_0 + \eta$ for $\eta > 0$ sufficiently small, these degree profiles will still all exist with high probability. To classify the vertices with the above degree profiles correctly, we must use the decision rule given by \eqref{eq:claimed-threshold}, and both $y$ and $z_1 z_2$ must be positive. Finally, since $4p(1-p) < 1$, there exist $\delta > \delta' > 0$ such that there are vertices with the required degree profiles \eqref{eq:profile-1} and \eqref{eq:profile-2}, and thus the proof follows.
\end{proof}

\bibliographystyle{abbrv}
\bibliography{main}

\appendix

\section{Supporting results for spectral algorithm analysis} \label{sec:appendix-supporting}
In this section, we complete the proofs of Lemmas~\ref{lemma:supporting-1},~\ref{lemma:supporting-2}, and~\ref{lemma:supporting-3}. 

\begin{proof}\emph{(Proof of Lemma \ref{lemma:supporting-1}).}
We use ideas from the proof of Theorem 9 in \cite{Hajek2016}. Throughout the proof, we condition on the realization of $\sigma_0$, without writing so explicitly. We first bound $\mathbb{E}\left[\Vert A - A^{\star}\Vert_2\right]$. Let $A'$ be an independent copy of $A$. By Jensen's inequality,
\[\mathbb{E}\left[\Vert A - A^{\star}\Vert_2 \right]=\mathbb{E}\left[ \Vert A - \mathbb{E}[A'] \Vert_2 \right] \leq \mathbb{E}\left[\Vert A - A'\Vert_2\right]. \]
Let $R$ be a matrix of iid Rademacher random variables, and let $\circ$ denote the elementwise product. Then
\[\mathbb{E}\left[\Vert A - A'\Vert_2\right] = \mathbb{E}\left[\Vert (A - A') \circ R\Vert_2\right] \leq 2 \mathbb{E}\left[\Vert A \circ R \Vert_2 \right].\]
Let $D$ be a matrix with the same distribution as $A\circ R$, so that $\mathbb{E}\left[\Vert A \circ R \Vert_2 \right] = \mathbb{E}\left[\Vert D \Vert_2 \right]$. Note that the entries of $D$ are independent, and the distribution of $D_{ij}$ depends on whether $\sigma_0(i) = \sigma_0(j)$. Let $B$ be the matrix where $B_{ij} = 1 - y$ if $D_{ij} = y$ and $B_{ij} = y - 1$ if $D_{ij} = -y$, and is equal to zero otherwise. By definition, the matrix $D+B$ has entries in $\{-1, 0,1\}$. For $i \neq j$, we have \[\mathbb{P}\left((D+B)_{ij} = 1 \right) = \mathbb{P}\left((D+B)_{ij} = - 1 \right) = \frac{1}{2} \alpha, ~~~ \mathbb{P}\left((D+B)_{ij} = 0 \right) = 1 -\alpha,\]
regardless of the communities of $i$ and $j$. The entries of $D+B$ are independent. We decompose $B$ into $B = B^{(p)} + B^{(q)}$, as follows. Set $B^{(p)}_{ij} = 0$ whenever $\sigma_0(i) \neq \sigma_0(j)$ or $i = j$. Otherwise,
\begin{align*}
B^{(p)}_{ij} = \begin{cases}1 - y & \text{with probability } \frac{1}{2}\alpha (1-p)\\
 y - 1 & \text{with probability } \frac{1}{2}\alpha (1-p)\\
 0 &\text{otherwise.}
\end{cases}
\end{align*}
Similarly, set $B^{(q)}_{ij} = 0$ whenever $\sigma_0(i) = \sigma_0(j)$. Otherwise,
\begin{align*}
B^{(q)}_{ij} = \begin{cases} 1 - y & \text{with probability } \frac{1}{2}\alpha (1-q)\\
y - 1 & \text{with probability } \frac{1}{2}\alpha (1-q)\\
0 &\text{otherwise.}
\end{cases}
\end{align*}
Note that $\mathbb{E}\left[B^{(p)}_{ij} \right] = \mathbb{E}\left[B^{(q)}_{ij} \right] = 0.$ 
Continuing,
\begin{align*}
\mathbb{E}\left[\Vert A \circ R \Vert_2 \right] &= \mathbb{E}\left[\Vert D \Vert_2 \right]\\
&=  \mathbb{E}\left[\left\Vert D + B - B^{(p)} - B^{(q)} \right\Vert_2 \right]\\
&\leq \mathbb{E}\left[ \Vert D + B \Vert_2 \right] + \mathbb{E}\left[\left\Vert  B^{(p)}  \right\Vert_2 \right] + \mathbb{E}\left[\left\Vert  B^{(q)}  \right\Vert_2 \right].
\end{align*}
Let $\overline{B}^{(p)}$ and  $\overline{B}^{(q)}$ be distributed as follows, independent of all other variables. Set $\overline{B}^{(p)}_{ij} = 0$ whenever $\sigma_0(i)  = \sigma_0(j)$ (note that this is the opposite condition as for $B^{(p)}$). Otherwise,
\begin{align*}
\overline{B}^{(p)}_{ij} = \begin{cases} 1 - y & \text{with probability } \frac{1}{2}\alpha (1-p)\\
y - 1 & \text{with probability } \frac{1}{2}\alpha (1-p)\\
 0 &\text{otherwise.}
\end{cases}
\end{align*}

Similarly, set $\overline{B}^{(q)}_{ij} = 0$ whenever $\sigma_0(i) \neq \sigma_0(j)$ or $i = j$. Otherwise,
\begin{align*}
\overline{B}^{(q)}_{ij} = \begin{cases} 1 - y & \text{with probability } \frac{1}{2}\alpha (1-q)\\
 y - 1 & \text{with probability } \frac{1}{2}\alpha (1-q)\\
 0 &\text{otherwise.}
\end{cases}
\end{align*}
Note that these distributions are symmetric and $\mathbb{E}\left[\overline{B}^{(p)}_{ij} \right] = \mathbb{E}\left[\overline{B}^{(q)}_{ij} \right] = 0$ for all $i, j$. Applying Jensen's inequality again,
\begin{align*}
\mathbb{E}\left[\Vert A \circ R \Vert_2 \right] &\leq \mathbb{E}\left[ \Vert D - B \Vert_2 \right] + \mathbb{E}\left[\left\Vert  B^{(p)} + \mathbb{E}\left[\overline{B}^{(p)} \right] \right\Vert_2 \right] + \mathbb{E}\left[\left\Vert  B^{(q)} + \mathbb{E}\left[\overline{B}^{(q)} \right] \right\Vert_2 \right]\\
&\leq \mathbb{E}\left[ \Vert D - B \Vert_2 \right] + \mathbb{E}\left[\left\Vert  B^{(p)} + \overline{B}^{(p)}  \right\Vert_2 \right] + \mathbb{E}\left[\left\Vert  B^{(q)} + \overline{B}^{(q)}  \right\Vert_2 \right].
\end{align*}
Let $X(r)$ denote the zero-diagonal matrix whose non-diagonal entries are i.i.d according to $\mu(r) = \frac{r}{2} \delta_1 + \frac{r}{2} \delta_{-1} + (1-r) \delta_0$. We have
\begin{align*}
\mathbb{E}\left[ \Vert D - B \Vert_2 \right] &= \mathbb{E}\left[ \Vert X(\alpha)\Vert_2 \right] \\
\mathbb{E}\left[\left\Vert  B^{(p)} + \overline{B}^{(p)}  \right\Vert_2 \right] &= (1-y) \mathbb{E}\left[\left \Vert X(\alpha (1-p))\right\Vert_2
\right]\\
\mathbb{E}\left[\left\Vert  B^{(q)} + \overline{B}^{(q)}  \right\Vert_2 \right] &= (1-y) \mathbb{E}\left[\left \Vert X(\alpha (1-q))\right\Vert_2
\right].
\end{align*}
For $r \geq c_0 \frac{\log n}{n}$,
\[\mathbb{E}\left[\Vert X(r) \Vert_2 \right] \leq c_2\sqrt{nr}, \]
for $c_2$ depending only on $c_0$ (see e.g. \cite[Theorem 9]{Hajek2016}). 
Putting everything together,
\begin{align*}
\mathbb{E}\left[\Vert A -A^{\star}\Vert_2\right] &\leq 2 \left[c_2(t)\sqrt{n \alpha} + (1-x)c_2\left(t(1-p) \right)\sqrt{n \alpha(1-p)} + (1-x)c_2\left(t(1-q) \right)\sqrt{n \alpha(1-q)}\right]\\
&\leq c_3(p,q,t) \sqrt{\log(n)},
\end{align*}
for some constant $c_3(p,q,t)$ depending only on $p$, $q$, and $t$. Applying Talagrand's concentration inequality for bounded Lipschitz functions as in the proof of \cite[Theorem 9]{Hajek2016} concludes the proof.
\end{proof}

\begin{proof}\emph{(Proof of Lemma \ref{lemma:supporting-2}).}
We follow the proof of Lemma 7 in \cite{Abbe2020}, deriving a Chernoff bound. Without loss of generality we may assume $\Vert w \Vert_{\infty} = 1$. Let $S_n = \sum_{i=1}^n w_i (X_i - \mathbb{E}[X_i])$. We prove the upper tail inequality first. For $\lambda, r > 0$,
\begin{align*}
\mathbb{P}\left(S_n \geq r \right) &= \mathbb{P}\left(e^{\lambda S_n} \geq e^{\lambda r} \right)\\
&\leq e^{-\lambda r}\mathbb{E}\left[e^{\lambda S_n} \right]\\
&= e^{-\lambda r} \prod_{i=1}^n \mathbb{E}\left[e^{\lambda w_i (X_i - \mathbb{E}[X_i])} \right].
\end{align*}
Computing the moment generating function,
\begin{align*}
\mathbb{E}\left[e^{\lambda w_i (X_i - \mathbb{E}[X_i])}\right] &= e^{-\lambda w_i(p_i - q_i y)}\mathbb{E}\left[e^{\lambda w_i X_i}\right]\\
&=e^{-\lambda w_i(p_i - q_i y)}\left(p_i e^{\lambda w_i} + q_i e^{-\lambda y w_i} + 1 - p_i - q_i \right).
\end{align*}
Taking the logarithm and applying $\log(1+x) \leq x$ for $x > -1$ and $e^x \leq 1 + x + \frac{e^z}{2}x^2$ for $|x| \leq z$,
\begin{align*}
&\log\left(\mathbb{E}\left[e^{\lambda w_i (X_i - \mathbb{E}[X_i])}\right] \right)\\
&= -\lambda w_i(p_i - q_i y) + \log\left(p_i e^{\lambda w_i} + q_i e^{-\lambda y w_i} + 1 - p_i - q_i  \right)\\
&\leq -\lambda w_i(p_i - q_i y) + p_i\left(e^{\lambda w_i} - 1 \right) + q_i\left(e^{-\lambda y w_i} - 1 \right)  \\
&\leq -\lambda w_i(p_i - q_i y) + p_i \left(\lambda w_i + \frac{e^{\lambda \Vert w \Vert_{\infty}}}{2} \left(\lambda w_i \right)^2 \right) + q_i \left(-\lambda y w_i + \frac{e^{\lambda y \Vert w \Vert_{\infty}}}{2} \left(\lambda y w_i \right)^2 \right)\\
&= \frac{\lambda^2 w_i^2}{2} \left(p_i e^{\lambda \Vert w \Vert_{\infty}} +  q_i y^2 e^{\lambda y \Vert w \Vert_{\infty}} \right)\\
&\leq \frac{\lambda^2 w_i^2 \max\{1, y^2\} \exp\left(\lambda \cdot \max\{1,y\} \Vert w \Vert_{\infty}\right)}{2}(p_i + q_i).
\end{align*}
Therefore, using $\Vert w \Vert_{\infty} = 1$, 
\begin{align*}
\log\left(\mathbb{P}\left(S_n \geq r \right) \right)&\leq -\lambda r + \frac{\lambda^2 \Vert w\Vert_2^2 \max\{1, y^2\} \exp\left(\lambda \cdot \max\{1,y\} \right)}{2} \max_i \{p_i + q_i\}.
\end{align*}
Choose 
\begin{align*}
\lambda &= \frac{1}{\max\{1,y\}}\left(1 \vee \log\left(\frac{\sqrt{n}}{\Vert w \Vert_2} \right) \right) > 0.
\end{align*}
We then have
\begin{align*}
\exp\left(\lambda \cdot \max\{1,y\} \right) &= e \vee \frac{\sqrt{n}}{\Vert w \Vert_2} \leq \frac{e \sqrt{n}}{\Vert w \Vert_2}.
\end{align*}
Note that $\Vert w \Vert_2 \leq \sqrt{n} \Vert w \Vert_{\infty} = \sqrt{n}$. Using $1 \vee \log x \leq \sqrt{x}$ for $x \geq 1$,
\begin{align*}
\lambda^2 \Vert w \Vert_2^2 \exp\left(\lambda \cdot \max\{1,y\} \right) &\leq \lambda^2 e \sqrt{n} \Vert w \Vert_2\\
&\leq \frac{1}{(\max\{1,y\})^2}\left(1 \vee \log\left(\frac{\sqrt{n}}{ \Vert w \Vert_2} \right)\right)^2 e \sqrt{n} \Vert w \Vert_2\\
&\leq \frac{e n}{\max\{1,y^2\}}.
\end{align*}
Substituting,
\begin{align*}
\log\left(\mathbb{P}\left(S_n \geq r \right) \right)&\leq -\lambda r + \frac{en}{2} \max_i\{p_i + q_i\}.
\end{align*}
Let $r = \lambda^{-1}(2 + \beta)n \cdot \max_i\{p_i + q_i\}$, so that
\begin{align*}
\log\left(\mathbb{P}\left(S_n \geq r \right) \right) &\leq  -(2+\beta) n \cdot \max_i\{p_i + q_i\} + \frac{en}{2} \max_i\{p_i + q_i\}\\
&\leq -\beta n \max_i\{p_i + q_i\}.
\end{align*}
By replacing $w$ with $-w$ we can obtain a lower tail bound. The proof is complete by a union bound.
\end{proof}
\begin{proof}\emph{(Proof of Lemma \ref{lemma:supporting-3}).}
The derivation is similar to Chernoff bound. 
Let $\lambda < 0$ to be determined later. By Markov's inequality, 
\begin{align*}
\PR \bigg(\sum_{i=1}^{n_1} W_i - \sum_{i=1}^{n_2} Z_i \leq \ve \log(n) \bigg) &= \PR \bigg(\exp\bigg(\lambda \sum_{i=1}^{n_1} W_i -\lambda\sum_{i=1}^{n_2}Z_i\bigg) \geq \exp\big(\lambda \ve \log(n)\big) \bigg)\\
&\leq n^{-\lambda \ve} \mathbb{E}\bigg[\exp\bigg(\lambda \sum_{i=1}^{n_1} W_i -\lambda\sum_{i=1}^{n_2}Z_i \bigg) \bigg].
\end{align*}
Using $\log(1+x) \leq x$ for $x > -1$, we have 
\begin{align*}
\log \Big(\mathbb{E}\big[ e^{\lambda W_i}\big] \Big) &= \log \left(e^{\lambda} \alpha p + e^{-\lambda y} \alpha (1-p) + 1 -\alpha \right)\\
&\leq e^{\lambda} \alpha p + e^{-\lambda y} \alpha (1-p)  -\alpha\\
&= \alpha \left(e^{\lambda} p + e^{-\lambda y}  (1-p)  -1\right).
\end{align*}
Similarly,
\[\log \Big(\mathbb{E}\big[ e^{-\lambda Z_i}\big] \Big) \leq \alpha \left(e^{-\lambda} q + e^{\lambda y} (1-q) -1 \right).\]
Therefore, using $n_1,n_2 = (1+o(1))\frac{n}{2}$, 
\begin{align*}
&\log \mathbb{P}\bigg(\sum_{i=1}^{n_1} W_i - \sum_{i=1}^{n_2} Z_i \leq \ve \log(n) \bigg)  \\
&\leq -\lambda \ve \log(n) + 
(1+o(1))\frac{n}{2} \alpha \left(e^{\lambda} p + e^{-\lambda y} (1-p) + e^{-\lambda} q + e^{\lambda y} (1-q) -2 \right)\\
&=  \log(n) \bigg[-\lambda \ve  + 
\frac{t}{2} \left(e^{\lambda} p + e^{-\lambda y} (1-p) + e^{-\lambda} q + e^{\lambda y} (1-q) -2 \right) +o(1) \bigg].
\end{align*}
Set $\lambda = \frac{1}{2}\log\big(\frac{q}{p} \big) < 0$. Observe that $e^{\lambda} = \sqrt{\frac{q}{p}}$, and 
\[e^{\lambda y} = \left(\frac{q}{p}\right)^{\frac{\log\left(\frac{1-q}{1-p} \right)}{2\log\left(\frac{p}{q} \right)}} = \sqrt{\frac{1-p}{1-q}}. \]
Therefore,
\begin{align*}
&\log \mathbb{P}\bigg(\sum_{i=1}^{\frac{n}{2}} W_i - \sum_{i=1}^{\frac{n}{2}} Z_i \leq \ve \log(n) \bigg)  \\
&\leq \log(n) \left[-\lambda \ve  + 
\frac{t}{2} \left(\sqrt{pq} + \sqrt{(1-p)(1-q)} + \sqrt{pq} + \sqrt{(1-p)(1-q)} -2 \right) +o(1) \right]\\
&= \log(n) \left[-\lambda \ve  - 
\frac{t}{2} \left(\left(\sqrt{p} - \sqrt{q} \right)^2 + \left(\sqrt{1-q} - \sqrt{1-p} \right)^2 \right) +o(1)\right]. 
\end{align*}
\end{proof}

\section{Proof of Poisson approximation}
\label{sec:appendix-2}
\begin{proof}\emph{(Proof of Lemma \ref{fact:stirling}).}
To prove the claim, observe that under the  assumptions on $m,m_1,m_2$, by Stirling's approximation and the fact that $1-\e^{-x} \asymp x$ as $x\to 0$, 
\begin{align*}
    \frac{m!}{(m-m_1 - m_2)!} &\asymp \frac{\e^{-m} m^{m+\frac{1}{2}}}{\e^{-m+m_1+m_2} (m-m_1-m_2)^{m-m_1-m_2+\frac{1}{2}}} \\
    &= \e^{-m_1-m_2} \frac{(m-m_1-m_2)^{m_1+m_2}}{(1-\frac{m_1+m_2}{m})^{m+\frac{1}{2}}}\\
    &\asymp \frac{(m-m_1-m_2)^{m_1+m_2}}{ (1-\frac{m_1+m_2}{m})^{\frac{1}{2}}}\\
    &\asymp \bigg(\frac{n}{2}\big(1+O(\log ^{-2}n )\big) \bigg) ^{m_1+m_2} 
    \\
    &\asymp 
\left(\frac{n}{2}\right)^{m_1 + m_2},
\end{align*}
where in the last step we have used 
$m_1,m_2 = o(\log ^{3/2} n)$ and the fact that $(1+x)^l \asymp 1+lx$ when $lx \to 0$. Also, 
\[(1-\alpha)^{m-m_1-m_2} \asymp e^{-\alpha(m-m_1-m_2)} \asymp e^{-t\log(n)/2}. \]
Thus, 
\begin{align*}
   \PR(N_a = m_1, N_b = m_2) &= \frac{m!}{m_1!m_2! (m-m_1-m_2)!} (\alpha p)^{m_1}(\alpha (1-p))^{m_2} (1-\alpha)^{m-m_1-m_2} \\
   & \asymp \e^{-\frac{t\log n}{2}}\frac{(\frac{t p \log n}{2})^{m_1}(\frac{t (1-p) \log n}{2})^{m_2}}{m_1!m_2!}, 
\end{align*}
and thus the proof follows. 
\end{proof}

\section{Proof of genie estimator characterization}\label{appendix-3}
\begin{proof}\emph{(Proof of Proposition~\ref{prop:genie-expression-general}).} 
Recall the definition of the degree profile above \eqref{eq:genie-lin-comb}, and also, recall from \eqref{genie-MLE-equal} that the MAP estimator equals the MLE under a uniform prior. 
Fix $u\in [n]$, and let $\sigma_1,\sigma_{-1}\in\cS$ be such that $\sigma_1(u) = +1, \sigma_{-1}(u) = -1$ and $\sigma_1(v) =\sigma_{-1}(v)$ for all $v\in [n] \setminus \{u\}$. 
For a fixed edge-labeled graph $g$, suppose $u$ has degree profile $D(\sigma_1,u)=(d_1(u),d_2(u),d_3(u),d_4(u))$. 
By definition, $D(\sigma_1,u) = D(\sigma_{-1},u)$. 
\begin{align*}
&\PR(G = g\mid \true = \sigma_1 ) \\
&= Z\times \binom{n_1(\sigma_1)-1}{d_1(u),d_2(u)}\binom{n_2(\sigma_{1})}{d_3(u),d_4(u)}(\alpha p_1)^{d_1(u)} (\alpha(1-p_1))^{d_2(u)} \\
& \hspace{7cm}(\alpha q)^{d_3(u)}  (\alpha (1-q))^{d_4(u)} (1-\alpha)^{n-1-\sum_id_i(u)},
\end{align*}
where $Z$ is not dependent on $u$. 
Similarly
\begin{align*}
&\PR(G = g\mid \true = \sigma_{-1} ) \\
&= Z\times \binom{n_1(\sigma_{-1})}{d_1(u),d_2(u)}\binom{n_2(\sigma_{-1})-1}{d_3(u),d_4(u)}(\alpha q)^{d_1(u)} (\alpha(1-q))^{d_2(u)} \\
& \hspace{7cm}(\alpha p_2)^{d_3(u)}  (\alpha (1-p_2))^{d_4(u)} (1-\alpha)^{n-1-\sum_id_i(u)}. 
\end{align*}
Moreover, $n_1(\sigma_{-1}) = n_1(\sigma_1)- 1$ and $n_2(\sigma_{-1}) = n_2(\sigma_1)+ 1$.
Therefore, the binomial coefficients above are also equal and 
\begin{align*}
    &\log \frac{\PR(G = g\mid \true = \sigma_{1})}{\PR(G = g\mid \true = \sigma_{-1} )} = d_1(u) \log \frac{p_1}{q} + d_2(u) \log \frac{1-p_1}{1-q} + d_3(u) \log \frac{q}{p_2} + d_4(u) \log \frac{1-q}{1-p_2}. 
\end{align*}
The genie estimator returns $+1$ if and only if the log likelihood ratio is non-negative. Hence, the proof of Proposition~\ref{prop:genie-expression-general} follows.  
\end{proof}
\end{document}